\documentclass{amsart}

\usepackage{amsmath}
\usepackage{mathtools}
\usepackage{amssymb}
\usepackage{color}
\usepackage{amsthm}
\usepackage{tikz}
\usepackage{tikz-cd}
\usetikzlibrary{positioning,calc}
\usepackage{hyperref}

\theoremstyle{plain}
\newtheorem{theorem}{Theorem}[section]
\theoremstyle{definition}
\newtheorem{definition}{Definition}[section]
\theoremstyle{plain}
\newtheorem{prop}{Proposition}[section]
\theoremstyle{plain}

\theoremstyle{plain}

\definecolor{Noir}{rgb}{0,0,0} 
\definecolor{Blanc}{rgb}{1,1,1} 
\definecolor{Gray}{rgb}{0.5,0.5,0.5} 
\definecolor{Rouge}{rgb}{0.8,0.1,0.1} 
\definecolor{DBleu}{RGB}{51,51,178} 
\definecolor{LBleu}{rgb}{0.85,0.85,1} 
\definecolor{Orange}{RGB}{255,140,0} 
 
\newcommand{\bcent}{\begin{center}} 
\newcommand{\ecent}{\end{center}} 
 
\newcommand{\benum}{\begin{enumerate}} 
\newcommand{\eenum}{\end{enumerate}} 
\newcommand{\bitem}{\begin{itemize}} 
\newcommand{\eitem}{\end{itemize}} 
 
\newcommand{\btab}{\begin{tabular}} 
\newcommand{\etab}{\end{tabular}} 
\newcommand{\beqn}{\begin{eqnarray}} 
\newcommand{\eeqn}{\end{eqnarray}} 
 
\newcommand{\bmath}{\begin{math}} 
\newcommand{\emath}{\end{math}} 
 
\newcommand{\noin}{\noindent} 

\providecommand{\tb}[1]{\textbf{#1}} 
\providecommand{\mb}[1]{\mathbf{#1}}

\newcommand{\ds}{\displaystyle}

\providecommand{\F}[1]{\mathbb{#1}}

\newcommand{\ZZ}{\mathbb{Z}}

\newcommand{\QQ}{\F{Q}} 
\newcommand{\RR}{\F{R}} 
\newcommand{\CC}{\F{C}}

\newcommand{\PP}{\F P}

\newcommand{\HH}{\mathbb H}

\newcommand{\plus}{\oplus} 
\newcommand{\Plus}{\bigoplus} 
\newcommand{\tensor}{\otimes}

\newcommand{\w}{\omega} 
 
\providecommand{\Cal}[1]{\mathcal{#1}} 
 
\newcommand{\CO}{\Cal O}

\newcommand{\calR}{\Cal R}
\newcommand{\calC}{\mathcal{C}}

\newcommand{\calM}{\mathcal{M}}
\newcommand{\calO}{\mathcal{O}}
\newcommand{\calE}{\mathcal{E}}
\newcommand{\calS}{\mathcal{S}}
\newcommand{\CD}{\Cal D}

\newcommand{\CH}{\Cal H}

\newcommand{\CM}{\Cal M} 
\newcommand{\CN}{\Cal N} 
 
\newcommand{\CP}{\Cal P} 
 
\newcommand{\ClC}{\Cal C}

\renewcommand{\epsilon}{\varepsilon}

\providecommand{\v}[1]{\vec{#1}}

\newcommand\varleq{\mathbin{\vcenter{\hbox{%
  \oalign{\hfil$\scriptstyle<$\hfil\cr 
          \noalign{\kern-.3ex} 
          $\scriptscriptstyle({-})$\cr}%
}}}} 
 
\renewcommand\subsetneq{\mathbin{\vcenter{\hbox{%
  \oalign{\hfil$\scriptstyle\subset$\hfil\cr 
          \noalign{\kern-.3ex} 
          $\scriptscriptstyle({-})$\cr}%
}}}} 

\author{Steven Rayan}
\address{Department of Mathematics \& Statistics\\
  McLean Hall, University of Saskatchewan\\
  Saskatoon, SK, Canada  ~S7N 5E6}
\email{rayan@math.usask.ca, ejs383@mail.usask.ca}

\author{Evan Sundbo}

\title[Twisted Argyle Quivers and Higgs Bundles]{Twisted Argyle Quivers and Higgs Bundles}
\date{\today}

\begin{document}

\maketitle

\noin\tb{Abstract.}  Ordinarily, quiver varieties are constructed as moduli spaces of quiver representations in the category of vector spaces.  It is also natural to consider quiver representations in a richer category, namely that of vector bundles on some complex variety equipped with a fixed sheaf that twists the morphisms.  Representations of $A$-type quivers in this twisted category --- known in the literature as ``holomorphic chains'' --- have practical use in questions concerning the topology of the moduli space of Higgs bundles.  In that problem, the variety is a Riemann surface of genus at least $2$, and the twist is its canonical line bundle.  We extend the treatment of twisted $A$-type quiver representations to any genus using the Hitchin stability condition induced by Higgs bundles and computing their deformation theory.  We then focus in particular on so-called ``argyle quivers'', where the rank labelling alternates between $1$ and integers $r_i\geq1$.  We give explicit geometric identifications of moduli spaces of twisted representations of argyle quivers on $\PP^1$ using invariant theory for a non-reductive action via Euclidean reduction on polynomials.  This leads to a stratification of the moduli space by change of bundle type, which we identify with ``collision manifolds'' of invariant zeroes of polynomials.   We also relate the present work to Bradlow-Daskalopoulos stability and Thaddeus' pullback maps to stable tuples.  We apply our results to computing $\QQ$-Betti numbers of low-rank twisted Higgs bundle moduli spaces on $\PP^1$, where the Higgs fields take values in an arbitrary ample line bundle.  Our results agree with conjectural Poincar\'e series arising from the ADHM recursion formula.\\

\tableofcontents

\section{Introduction}

By now, quiver varieties are household names in mathematics, particularly when it comes to representation theory and geometry.  Most often they are formed by labelling the nodes of a directed graph with nonnegative integers and then considering linear representations up to isomorphism.  Here, we are referring to representations in the category of vector spaces: to each node, we assign a vector space (over $\CC$ or some other algebraically-closed field) of the prescribed dimension, and a corresponding linear map to each arrow.  Another common construction of quiver varieties is the Nakajima construction \cite{HK:94,HK:96}.  This also uses the category of vector spaces, but each arrow in the initial quiver is ``doubled'' by adding an arrow with the opposite orientation.  These are interpreted as cotangent directions to the space of representations of the original quiver.  The quiver variety is then constructed as a hyperk\"ahler quotient \cite{HKLR:87}, and thus provides an important class of examples of Calabi-Yau manifolds.  Broadly speaking, quiver varieties have applications to toric geometry, vertex algebras, noncommutative geometry, integrable systems, and gauge theories.

Quiver varieties can also be constructed in other categories.  The natural generalization of the category of vector spaces is the category $\mbox{Bun}(X)$ of vector bundles on a fixed complex variety $X$.  Here, we label each node with two integers, $r_i$ and $d_i$, where $r_i\geq0$.  The $r_i$ numbers specify the ranks of the corresponding vector bundles while the $d_i$'s fix their respective degrees.  If one node is represented by $U_i$ and another is represented by $U_j$ and there is an arrow between them, then we assign to that arrow a vector bundle morphism in $\mbox{Hom}(U_i,U_j)$.  One of the main applications of working in this setting is to the study of Higgs bundles.  A \emph{Higgs bundle} is a vector bundle $E\to X$ together with a regular $\CO_X$-linear map $\Phi:E\to E\tensor\w_x$ satisfying the integrability condition $\Phi\wedge\Phi=0$, where $\w_X$ is the canonical line bundle of $X$ \cite{NJH:86,TSCH:92}.  If we consider a ``twisted'' category $\mbox{Bun}(X,\w_X)$ where the morphisms are twisted by $\w_X$, we can pose the following alternative but equivalent definition:  A Higgs bundle is a representation of rank $r$ and degree $d$ of a single node, labelled $r,d$, with a loop.

What makes this useful and not just a ``rebranding'' is that the moduli space of Higgs bundles on $X$ comes with a natural linear action of $\CC^\times$ that rescales $\Phi$.  The cohomology of the moduli space localizes to the components of the fixed-point set \cite{NJH:86,TSCH:91,PBG:94,HT:03,GR:15}, which themselves are positive-dimensional due to the noncompactness of the moduli space.  The components are indexed by partitions of $r$ and $d$ into some number of parts that are strung out into an $A$-type quiver.  The fixed points themselves are representations of these quivers in $\mbox{Bun}(X,\w_X)$, and so the problem of computing topological invariants (e.g. rational Betti numbers) of the moduli space becomes a question about invariants of quiver varieties of $A$-type in $\mbox{Bun}(X,\w_X)$ \cite{PBG:95,SR:11,SR:17}.

Such representations of $A$-type quivers in a twisted category of bundles have acquired the name ``holomorphic chains'' and, in the particular case of the quiver $A_2$, ``holomorphic triples''.  When the quiver is not specifically $A$-type, the nomenclature ``quiver bundle'' has been used.  Given their importance to Higgs bundles, various aspects of the topology, geometry, and homological algebra of holomorphic chains and quiver bundles have been explored in recent years: \cite{PBG:95,PBG:01,AG:01,AG:03,BGG:04,GK:05,AS:05,AGS:06,GGM:07,AS:08,SR:11,AS:12,AS:13,GHS:14,SM:16,AS:17,SR:17,MS:17,BGGH:17,GN:17}, to name a few.

Higgs bundles are most often considered over algebraic curves --- where the integrability condition is trivial --- of genus $2$ or larger.  We will also restrict to representations of $A$-type quivers over complex algebraic curves, but make two generalizations: we consider arbitrary genus and we replace $\w_X$ with an arbitrary line bundle of nonnegative degree.  Given the connection to localization of Higgs bundles, we use the notion of stability induced by the Hitchin stability condition on Higgs bundles \cite{NJH:86}.

Our first result is a calculation of the generic dimension of the moduli space of representations of such a quiver (Theorem \ref{ThmDimn}) using spectral sequences and hypercohomology for the differential induced by the morphisms themselves, in concert with the \v Cech differential.  This is for any genus, any twisting line bundle $L$, and any $n$, where $n$ is the length of the quiver $A_n$.

Next, we specialize to a particular configuration for which we are able to obtain concrete results on the geometry of the moduli space:

\begin{definition}
An $A$-type quiver of the form
\begin{equation}\nonumber
 \bullet_{1,d_1} \longrightarrow \bullet_{r_2,d_2}\longrightarrow\bullet_{1,d_3}\longrightarrow\cdots\longrightarrow\bullet_{r_{n-1},d_{n-1}}\longrightarrow \bullet_{1,d_n}
\end{equation}
is called an \emph{argyle quiver}.  
\end{definition}

Our first result concerning this shape of quiver is Theorem \ref{ThmPullback}, which generalizes the relationship between holomorphic triples (representations of argyle quivers of length $n=2$, \cite{PBG:95,BGG:04}) and ``stable pairs'' \cite{MT:94} to general argyle quivers in the form of a generalized pullback diagram.  As part of this result, we show that the Hitchin stability condition always has a corresponding Bradlow-Daskalopoulos stability parameter (cf. \cite{BD:91}) and vice-versa, a correspondence that boils down to the fact that a particular linear system always has a unique solution.

While the twisted Higgs bundle moduli space can be constructed as a geometric-invariant-theoretic (GIT) quotient for a reductive group \cite{NN:91} (cf. also \cite{BGL:11}), the argyle quiver varieties themselves are non-reductive quotients, despite sitting as subvarieties in various twisted Higgs moduli spaces.  When $X=\PP^1$, the non-reductive contributions from the automorphism group act in a particular way, namely via Euclidean reduction on spaces of polynomials.  For this reason, we specialize in Section \ref{Sect1k1} to $X=\PP^1$ in order find explicit geometric identifications of the varieties associated to argyle quivers.  Our first result in this direction is Theorem \ref{Thm1k1}, which concerns argyle quivers of length $n=3$, which we also refer to as type $(1,k,1)$ to reflect the prescribed ranks of the three nodes.  Here, we remove the ``collision locus'' where zeroes of maps in a representation that are invariant under the action of automorphisms become coincident.  We fix the holomorphic type of the central bundle and then identify the projective closure of the collision-free subvariety of the moduli space with products of projective spaces and Grassmannians.  From this, we are able to state Theorem \ref{ThmArgyle}, which describes how to compute the moduli space associated to an arbitrary argyle quiver from $(1,k,1)$ pieces.  Finally, we give a number of examples of how strata are glued together by identifying $(1,k,1)$ quiver varieties for different holomorphic types as collision varieties of one another.

We conclude the present work with applications to the topology of twisted Higgs bundle moduli spaces on $\PP^1$ with a complete account of the rational Betti numbers in rank $2$ and any twisting line bundle $L=\CO(t)$, as well as examples in rank $3$.\\

\noin\emph{Acknowledgements.}  We thank Steven Bradlow, Jonathan Fisher, Peter Gothen, Laura Fredrickson, and Sergey Mozgovoy for useful conversations related to this work.   The first named author acknowledges the support of an NSERC Discovery Grant and a University of Saskatchewan New Faculty Recruitment Grant.  The second named author acknowledges the Department of Mathematics \& Statistics at the University of Saskatchewan for a Graduate Teaching Fellowship.  Both authors thank Laura Schaposnik for her hospitality and useful discussions during the Geometry and Physics of Higgs Bundles workshops at UIC, for which we gratefully acknowledge travel support from NSF grants DMS 1107452, 1107263, 1107367 ``RNMS: GEometric structures And Representation varieties (GEAR)''.

\section{Deformation theory}\label{SectDefs}

We write $\ClC = (U_1,\dots,U_n;\phi_1,\dots,\phi_{n-1})$ for a representation of a quiver $$Q=\bullet_{r_1,d_1}\longrightarrow\bullet_{r_2,d_2}\longrightarrow\dots\longrightarrow\bullet_{r_n,d_n}$$in the category $\mbox{Bun}(X,L)$, where $X$ is a fixed complex projective variety and $L$ is any sheaf on $X$, which we use to twist morphisms between bundles.  The morphisms are graded by exterior powers of $L$: $\mbox{Hom}^k(U,V)=\mbox{Hom}(U,V\tensor\wedge^kL)$.   We denote by $\calM_{X,L}(Q)$ the moduli space of representations of $Q$ in $\mbox{Bun}(X,L)$.  This is the set of representations $(U_1,\dots,U_n;\phi_1,\dots,\phi_{n-1})$ up to the following equivalence: $\phi_i \sim \phi_i'$ if there exists $\sigma\in\text{Aut}(U_{i+1})$ such that $\phi_i =\sigma\phi_i'$, and that $U_i \sim U'_i$ if there exists $\tau\in\text{Aut}(U_{i})$ such that $g_{\alpha\beta} = \tau_{\alpha\beta}^{-1}g'_{\alpha\beta}\tau_{\alpha\beta}$, where $g_{\alpha\beta}$ and  $g'_{\alpha\beta}$ are the transition functions of $U_i$ and $U'_i$ respectively. Then $\left(U_1,\ldots,U_n;\phi_1,\ldots,\phi_{n-1}\right)\sim\left(U_1',\ldots,U_n';\phi_1',\ldots,\phi_{n-1}'\right)$ if and only if $\phi_i \sim \phi_i'$ for all $i=1,\ldots,n-1$, and $U_i \sim U_i'$ for all $i=1,\ldots,n$.  Given a representation $\ClC$, let $\Phi_\ClC$ be the twisted endomorphism $E_\ClC\to E_\ClC\tensor L$ defined by $\Phi_\ClC=\sum\phi_i$, where $E_\ClC=\Plus U_i$.  The notion of stability in play is analagous to the usual Hitchin stability condition for Higgs bundles \cite{NJH:86}:

\begin{definition}
A subrepresentation $\CD$ of $\ClC$ is \emph{invariant} if $\Phi_\ClC(E_\CD)\subseteq E_\CD\otimes L$.  The \emph{slope} of $\CD$ is $\mu(\CD) = \ds\frac{\text{deg}(E_\CD)}{\text{rank}(E_\CD)}$.   A representation $\ClC$ is \emph{semistable} if $\mu(\CD) \leq \mu(\ClC)$ for all nonzero, proper invariant subrepresentations $\CD$ of $\ClC$.  We denote $\mu(\ClC)$ by $\mu_{tot}$.  If this inequality fails for some $\CD$ we say that the representation is \emph{unstable}, and if the inequality is strict for all such $\CD$,  then we say that $\ClC$ is \emph{stable}.
\end{definition}

We assume $r=\text{rank}(E_\ClC)$ and $d=\text{deg}(E_\ClC)$ coprime throughout so that the semistable representations are exactly the stable representations.  We note that the stability condition for quiver representations is not rigid: there are infinitely many stability conditions parametrized by $\RR^n$, where $n$ is the number of nodes in $Q$ (e.g. \cite{PBG:95}).  Because of the applications we have in mind, we choose the condition corresponding to Higgs bundles.

We begin by calculating the dimension of the tangent space $T_{\calC}\calM_{X,L}(Q)$, and hence the expected dimension of the moduli space, where the underlying graph of $Q$ is $A_n$ for some $n$ and where $L$ is any line bundle on a complex projective curve $X$ of any genus. We do this by calculating a certain hypercohomology $\HH^1$ at a point in $\calM_{X,L}(Q)$.  We note that our result in this section can be deduced from more general arguments in \cite{AS:08}.  In parallel, parabolic versions of this result can be found in \cite{BY:96,GGM:07}.

Given a choice of $\left(U_1,\ldots,U_n\right)$, we define vector spaces $V^{p,q}$  by
\begin{equation}
V^{p,q} =  H^p\left(\left(\bigoplus_{i=1}^{n-q}U_i^*\tensor U_{i+q}\right)\tensor\wedge^qL\right).\nonumber
\end{equation}
It is worth noting that $(\phi_1,\ldots\phi_{n-1})\in V^{0,1}$. We also define a differential $\delta_{\Phi}:V^{p,q}\to V^{p,q+1}$ by the following:
\begin{equation}\begin{split}
\delta_\Phi (\psi_1,&\ldots,\psi_{n-q}) =\\& (\psi_{2}\phi_1-\phi_{1+q}\psi_1,\mbox{ }\psi_{3}\phi_2-\phi_{2+q}\psi_2,\ldots,\psi_{n-q}\phi_{n-(q+1)}-\phi_{n-1}\psi_{n-(q+1)})\nonumber
\end{split}\end{equation}
($\delta_\Phi$ is named for its dependence on the total map $\Phi:=\bigoplus_{i=1}^{n-1} \phi_i$).
Now we have given $\bigoplus_{p,q} V^{p,q}$ the structure of a bi-graded Lie algebra, with $\delta_\Phi(-)$ being the Lie bracket.  The hypercohomology $\HH^1$ that we are looking for fits into an exact sequence
\begin{equation}
0\longrightarrow\calE^{0,1}_2 \longrightarrow \HH^1 \longrightarrow \calE^{1,0}_2 \longrightarrow \calE^{0,2}_2 \longrightarrow \HH^2\nonumber
\end{equation}
where
\begin{equation}
\calE^{p,q}_2 =  \frac{\text{ker}\big(V^{p,q} \xrightarrow[]{\mbox{ }\delta_\Phi\mbox{ }}V^{p,(q+1)}\big)}{\text{im}\big( V^{p,(q-1)} \xrightarrow[]{\mbox{ }\delta_\Phi\mbox{ }}V^{p,q}\big)}.\nonumber
\end{equation}
If $\HH^2 = 0$, then we say that the deformations are \emph{unobstructed}, and in this case it is clear that $$T_{\calC}\calM_{X,L}(Q) = \HH^1(\calC)=\calE^{0,1}_2 \oplus \calE^{1,0}_2.$$

\begin{prop}\label{HH^2=0}
If $X$ is a Riemann surface and $L$ is a line bundle, then the deformations of $\calM_{X,L}(Q)$ are unobstructed.
\end{prop}
\begin{proof}
By the usual filtration, the space $\HH^2$ consists of contributions from three spaces (which we will show are all trivial): $\calE_2^{0,2}$, $\calE_2^{2,0}$, and $\calE_2^{1,1}$.  

Firstly, $L$ is a line bundle on a projective algebraic curve so $\wedge^2L = 0$. Therefore $\calE_2^{0,2}$, whose numerator consists of the kernel of $$H^0((U_1^*U_3\oplus\ldots\oplus U^*_{n-2}U_n)\tensor\wedge^2L)$$ under $\delta_\Phi$, is zero.

To see that $\calE_2^{2,0}$ is also zero, note that its numerator is the kernel of $$H^2(\text{End}U_1\oplus\ldots\oplus \text{End} U_n)$$ under $\delta_\Phi$, which is trivial on a curve.

Finally, we must deal with
\begin{equation}\begin{split}
\calE_2^{1,1} &= 
 \frac{\text{ker}\big(H^1((U_1^* U_2\oplus\ldots\oplus U^*_{n-1}U_n)\tensor L)\xrightarrow[]{\mbox{ }\delta_\Phi\mbox{ }} H^1((U_1^* U_3\oplus\ldots\oplus U^*_{n-2}U_n)\tensor \wedge^2L)\big)}{\text{im}\big(H^1\big(\text{End}U_1\oplus\ldots\oplus \text{End} U_n\big)\xrightarrow[]{\mbox{ }\delta_\Phi\mbox{ }}H^1((U_1^* U_2\oplus\ldots\oplus U^*_{n-1}U_n)\tensor L)\big)} \\
&= \frac{H^1((U_1^* U_2\oplus\ldots\oplus U^*_{n-1}U_n)\tensor L)}{\text{im}\big(H^1\big(\text{End}U_1\oplus\ldots\oplus \text{End} U_n\big)\xrightarrow[]{\mbox{ }\delta_\Phi\mbox{ }}H^1((U_1^* U_2\oplus\ldots\oplus U^*_{n-1}U_n)\tensor L)\big)}. \nonumber
\end{split}\end{equation}
If we can show that the map in the denominator is surjective, then we will have shown that $\calE_2^{1,1}$ is trivial.  To do this, consider the Serre-dual map
\begin{equation}
H^0((U_1^* U_2\oplus\ldots\oplus U^*_{n-1}U_n)^*\tensor L^* \tensor \omega)\xrightarrow[]{\mbox{ }\delta^*_\Phi\mbox{ }}H^0( (\text{End}U_1\oplus\ldots\oplus \text{End} U_n)^* \tensor \omega)\nonumber
\end{equation}
where $\omega$ is the canonical line bundle on $X$.  This map is equivalent to
\begin{equation}
H^0(U_1 U^*_2\oplus\ldots\oplus U_{n-1}U^*_n)\xrightarrow[]{\mbox{ }\delta^*_\Phi\mbox{ }}H^0( (\text{End}U_1\oplus\ldots\oplus \text{End} U_n) \tensor L).\nonumber
\end{equation}
The map $\delta^*_\Phi$ is injective if and only if $\delta_\Phi$ is surjective, and vice versa. We can calculate
\begin{equation}\begin{split}
\delta_{\Phi}^* (\eta_1,&\ldots,\eta_{n-1}) = \\&(\phi_1^*\eta_1^*,\mbox{ } \phi_2^*\eta_2^* - \eta_1^*\phi_1^*,\ldots,\phi_{n-1}^*\eta_{n-1}^* - \eta_{n-2}^*\phi_{n-2}^*,-\eta_{n-1}^*\phi_{n-1}^*)\nonumber
\end{split}\end{equation}
from which it is clear that the kernel of $\delta^*_{\Phi}$ is trivial, and thus $\delta_\Phi^*$ is injective and $\delta_\Phi$ is surjective.  Now the image of $$H^1(\text{End}U_1\oplus\ldots\oplus \text{End} U_n)$$ is $$H^1((U_1^* U_2\oplus\ldots\oplus U^*_{n-1}U_n)\tensor L)$$ and this tells us that $\calE_2^{1,1} = 0$, as required.

\end{proof}

This proposition gives us an inroad to calculating the expected dimension of $\calM_{X,L}(Q)$ (defined as the dimension of $\HH^1(U_1,\ldots,U_n;\phi_1,\ldots,\phi_{n-1})$). In particular,
\begin{equation}
\text{dim}\calM_{X,L}(Q) = e^2_{0,1}+ e^2_{1,0}\nonumber
\end{equation}
where
\begin{equation}\label{e01}
e^2_{0,1} =\text{dim}\left(
 \frac{H^0\big((U_1^* U_{2}\oplus\ldots\oplus U_{n-1}^*U_n)\tensor L\big)}{\text{im}\big(H^0\big( \text{End}U_1\oplus\ldots\oplus \text{End} U_n\big)\xrightarrow[]{\mbox{ }\delta_\Phi\mbox{ }}H^0\big((U_1^* U_{2}\oplus\ldots\oplus U_{n-1}^*U_n)\tensor L\big)\big)} \right)
\end{equation}
and
\begin{equation}\label{e10}
e^2_{1,0} = \text{dim}\left(\text{ker}\big(H^1\big( \text{End}U_1\oplus\ldots\oplus \text{End} U_n\big)\xrightarrow[]{\mbox{ }\delta_\Phi\mbox{ }} H^1\big((U_1^* U_{2}\oplus\ldots\oplus U_{n-1}^*U_n)\tensor L\big)\big)\right).
\end{equation}

\begin{theorem}\label{ThmDimn}
Given a quiver 
\begin{equation}
Q = \bullet_{r_1,d_1} \longrightarrow \bullet_{r_2,d_2}\longrightarrow\cdots\longrightarrow \bullet_{r_n,d_n}\nonumber
\end{equation}
the dimension of the moduli space of representations in the category of $L$-twisted vector bundles (with $L$ of degree $t$) over an algebraic curve $X$ of genus $g$  is
\begin{equation}\begin{split}\nonumber
 \sum_{i=1}^{n-1} \big(r_id_{i+1} -r_{i+1}d_i + r_ir_{i+1}t\big)+(1-g)\left(\sum_{i=1}^{n-1}r_ir_{i+1}-\sum_{i=1}^n r_i^2\right)+\min_{1\leq i \leq n}\{h^0(\emph{End}U_i)\}.
\end{split}\end{equation}
\end{theorem}

\begin{proof}
We have from Proposition~\ref{HH^2=0} that $\text{dim}\calM_{X,L}(Q) = e^2_{0,1}+ e^2_{1,0}$, where $e^2_{0,1}$ and $e^2_{1,0}$ are given by equations (\ref{e01}) and (\ref{e10}), respectively.  We also showed in the proof that the map $$H^1\big( \text{End}U_1\oplus\ldots\oplus \text{End} U_n\big)\xrightarrow[]{\mbox{ }\delta_\Phi\mbox{ }} H^1\big((U_1^* U_{2}\oplus\ldots\oplus U_{n-1}^*U_n)\tensor L\big)$$ is surjective; thus we can say
\begin{equation}
e^2_{1,0}=h^1\big( \text{End}U_1\oplus\ldots\oplus \text{End} U_n\big) - h^1\big((U_1^* U_{2}\oplus\ldots\oplus U_{n-1}^*U_n)\tensor L\big)\nonumber
\end{equation}
and a similar argument will allow us to analyze $e^2_{0,1}$. We would like to say that $$ H^0\big( \text{End}U_1\oplus\ldots\oplus \text{End} U_n\big)\xrightarrow[]{\mbox{ }\delta_\Phi\mbox{ }} H^0\big((U_1^* U_{2}\oplus\ldots\oplus U_{n-1}^*U_n)\tensor L\big)$$ was injective, but this is not quite true.  By inspecting the map
\begin{equation}\begin{split}
\delta_\Phi (\psi_1,&\ldots,\psi_{n}) =\\& (\psi_{2}\phi_1-\phi_{1}\psi_1,\mbox{ }\psi_{3}\phi_2-\phi_{2}\psi_2,\ldots,\psi_{n}\phi_{n-1}-\phi_{n-1}\psi_{n-1})\nonumber
\end{split}\end{equation}
it can be seen that it is injective only if we except one of the terms $\psi_i$.  Ignoring an arbitrary $\psi_i$ would result in a map that was injective but may not have an image of the same dimension as the full $\delta_\Phi$ map.  We must ignore a $\psi_i$ coming from $H^0(\text{End}U_i)$ of minimal dimension.  If there are more than one having the same (minimal) dimension, then it will not matter which we remove, as the resulting dimensions will be the same.   That is, we can think of $\delta_\Phi$ as being $\min_{1\leq i \leq n}\{h^0(\text{End}U_i)\}$-far away from being injective.  This tells us that
\begin{equation}
e^2_{0,1}  = h^0\big((U_1^* U_{2}\oplus\ldots\oplus U_{n-1}^*U_n)\tensor L\big)
-h^0\big( \text{End}U_1\oplus\ldots\oplus \text{End} U_n\big)+\min_{1\leq i \leq n}\{h^0(\text{End}U_n)\}.\nonumber
\end{equation}

Now we can apply Riemann-Roch to $e^2_{0,1}+ e^2_{1,0}$ to obtain 
\begin{equation}\begin{split}\label{e01+e10}
&e^2_{0,1}+ e^2_{1,0} \\
&\qquad
=\text{deg}\big((U_1^* U_{2}\oplus\ldots\oplus U_{n-1}^*U_n)\tensor L\big)+ \text{rank}\big((U_1^* U_{2}\oplus\ldots\oplus U_{n-1}^*U_n)\tensor L\big)(1-g)
\\
& \qquad\qquad- \text{deg}\big( \text{End}U_1\oplus\ldots\oplus \text{End} U_n\big)- \text{rank}\big( \text{End}U_1\oplus\ldots\oplus \text{End} U_n\big)(1-g)
\\
&\qquad\qquad\qquad+\min_{1\leq i \leq n}\{h^0(\text{End}U_n)\}
\\
&\qquad=\sum_{i=1}^{n-1}\text{deg}(U_i^*U_{i+1}L) +(1-g)\sum_{i=1}^{n-1}\text{rank}(U_i^*U_{i+1}L) \\&\qquad\qquad-\sum_{i=1}^{n}\text{deg}(\text{End}(U_i)) - (1-g)\sum_{i=1}^{n}\text{rank}(\text{End}(U_i)) +\min_{1\leq i \leq n}\{h^0(\text{End}U_n)\}
\\
&\qquad=\sum_{i=1}^{n-1}\text{deg}(U_i^*U_{i+1}L) + (1-g)\Big(\sum_{i=1}^{n-1}r_ir_{i+1}-\sum_{i=1}^n r_i^2\Big)+\min_{1\leq i \leq n}\{h^0(\text{End}U_n)\}.
\end{split}\end{equation}
It remains to calculate $\text{deg}(U_i^*U_{i+1}L)$. Note that the following calculation also serves to demonstrate that the dimension of the moduli space only depends on the degrees and ranks of the $U_i$, not on their specific structures (how they may split, etc.). We decompose the determinant of $U_i^*U_{i+1}L$ as follows:
\begin{equation}\begin{split}
\text{det}(U_i^*U_{i+1}L) &= \text{det}(U_i^*)^{\otimes r_{i+1}} \otimes\text{det}(U_{i+1}L)^{\otimes r_i}
\\
&=\text{det}(U_i^*)^{\otimes r_{i+1}} \otimes\text{det}(U_{i+1})^{\otimes r_i}\otimes\text{det}(L)^{\otimes r_ir_{i+1}}\nonumber
\end{split}\end{equation}
thus
\begin{equation}\begin{split}
\text{deg}(U_i^*U_{i+1}L) 
&=\text{deg}(\text{det}(U_i^*U_{i+1}L))
\\
&=r_{i+1}\text{deg}(U_i^*) + r_i\text{deg}(U_{i+1}) + r_ir_{i+1}\text{deg}(L)
\\
&= r_id_{i+1}-r_{i+1}d_i+r_ir_{r+1}t.\nonumber
\end{split}\end{equation}
This calculation along with equation~(\ref{e01+e10}) gives the result.
\end{proof}

\section{Pullback diagrams and stability}\label{SectPullStab}

\subsection{Stable tuples}
Given an argyle quiver $Q$ of length $n=2q+1$, we will need to consider a space of \emph{stable $4q$-tuples}, analagous to the stable pairs studied by Thaddeus.  Stability for these tuples depends on $2q$ parameters, which we denote by $\sigma=(\sigma_1,\ldots,\sigma_{2q})\in\RR^{2q}$.  Accordingly, we will define a space
\begin{equation}
\calR^{\sigma}_{X}(k_1,\dots,k_{2q};e_1,\dots,e_{2q}),\nonumber
\end{equation}
that parametrizes stable tuples of the form $\big\{(V_1,\ldots,V_{2q};\phi_1,\ldots,\phi_{2q})\big\}$, where $V_i$ is a vector bundle of rank $k_i$ and degree $e_i$ (with $k_i=1$ whenever $i$ is even), $\phi_i \in H^0(X,V_i)$ for $i$ odd, and $\phi_i \in H^0(X,V_{i-1}^*V_i)$ for $i$ even.

The stability condition arising from the choice of $\sigma$ follows from the well-known $\alpha$-stability condition on the space of holomorphic chains  (equivalently, the moduli space of representations of the quiver $Q$).  This space is 
\begin{equation}
\calM_X^\alpha(r_1,\ldots,r_{n};d_1,\ldots,d_n) =  \big\{(U_1,\ldots,U_n;\phi_1,\ldots,\phi_{n-1})\big\}.\nonumber
\end{equation}
with rk$(U_i)= r_i$, $\text{deg}(U_i) = d_i$, and $\phi_i \in H^0(X,U_i^*U_{i+1}L)$.  The $\alpha$-slope of a holomorphic chain $\calC =(U_1,\ldots,U_n;\phi_1,\ldots,\phi_{n-1})$ depends on the $2q$-tuple $\alpha = (\alpha_1,\ldots,\alpha_{n-1}) \in \RR^{2q}$, and is defined as 
\begin{equation}
\mu_{\alpha}(\calC) = \frac{\sum_{i=1}^n d_i+ \sum_{i=1}^{n-1}\alpha_ir_{i+1}} {\sum_{i=1}^n r_i}\nonumber
\end{equation}

We say that a holomorphic chain $\calC\in \calM^\alpha_X(r_1,\ldots,r_{n};d_1,\ldots,d_n)$ is $\alpha$-stable if $\mu_{\alpha}(\calC') < \mu_{\alpha}(\calC)$ for each proper, $(\phi_1\oplus\ldots\oplus\phi_{n-1})$-invariant subchain $\calC'\subset \calC$.  Now we will play with this a little bit: recall the usual slope $\mu(\calC) = \frac{d}{r}$ and set
\begin{equation}\begin{split}
\alpha_i = \frac{r}{r_{i+1}}\big(\sigma_i - \frac{1}{n-1}\mu(\calC)\big)\nonumber
\end{split}\end{equation}
for all $i = 1,\ldots, {n-1}$. Now the expression $\mu_{\vec{\alpha}}(\calC') < \mu_{\vec{\alpha}}(\calC)$ becomes
\begin{equation*}
\frac{d' + \sum_{i=1}^{n-1}r'_{i+1}\frac{r}{r_{i+1}}\big(\sigma_i - \frac{1}{n-1}\mu(\calC)\big) }{r'} < \frac{d +  \sum_{i=1}^{n-1}r_{i+1}\frac{r}{r_{i+1}}\big(\sigma_i - \frac{1}{n-1}\mu(\calC)\big)}{r}
\end{equation*}
\begin{equation*}
\mu(\calC') +  \sum_{i=1}^{n-1}\frac{r}{r'}\frac{r'_{i+1}}{r_{i+1}}\big(\sigma_i - \frac{1}{n-1}\mu(\calC)\big) < \mu(\calC) + \sum_{i=1}^{n-1}\big(\sigma_i - \frac{1}{n-1}\mu(\calC)\big) 
\end{equation*}
\begin{equation*}
\mu(\calC') < \sum_{i=1}^{n-1}\frac{r}{r'}\frac{r'_{i+1}}{r_{i+1}}\big( \frac{1}{n-1}\mu(\calC)-\sigma_i\big)+\sum_{i=1}^{n-1}\sigma_i
\end{equation*}
\begin{equation}\label{stablepent}
d' < \frac{d}{n-1}\sum_{i=1}^{n-1}\frac{r'_{i+1}}{r_{i+1}} +\sum_{i=1}^{n-1}\sigma_i \left(r'-r\frac{r'_{i+1}}{r_{i+1}}\right).
\end{equation}

This is the $\sigma$-stability condition for holomorphic chains of length $n$.  To specialize to $\calR^{\sigma}_{X}(k_1,\dots,k_{2q};e_1,\dots,e_{2q})$, we need to focus on chains of the form$$\calC=(\calO,U_2,\ldots,U_n;\phi_1,\ldots,\phi_{n-1}),$$where $U_i$ is a line bundle for each $i$ odd.  This can be viewed as a $4q$-tuple in $\calR^{\sigma}_{X}(k_1,\dots,k_{2q};e_1,\dots,e_{2q})$ for $k=r-1$, $k_i = r_{i+1}$, and $e_i = d_{i+1}$.  Since we have set $d_1 = d'_1 = 0$, we have $d=e$ and can write~(\ref{stablepent}) as
\begin{equation}
e' < \frac{e}{2q}\sum_{i=1}^{2q}\frac{k'_i}{k_i} + \sum_{i=1}^{2q}\sigma_i\left(r'-(k+1)\frac{k'_i}{k_i}\right).\nonumber
\end{equation}
Now, this expression still depends explicitly on $r'_1$, which is certainly strange if we are trying to look at this as a stability condition for a $4q$-tuple.  To remedy this, we note that if $\phi_1\in H^0(X,U'_2)\setminus\{0\}$, then it is clear that $r'_1 = 1$.  Conversely, if $\phi_1 \not\in H^0(X,U'_2)\setminus\{0\}$, then we must have $r'_1=0$.  Hence, the stability condition on $\calR^{\sigma}_{X}(k_1,\dots,k_{2q};e_1,\dots,e_{2q})$ settles nicely into two cases:

\begin{definition}
A $4q$-tuple $(V_1,\ldots,V_{2q};\phi_1,\ldots,\phi_{2q})$ with $\text{rk}(V_i) = r_i$ and $\text{deg}(V_i) = e_i$ is \emph{stable} if for every sub-$4q$-tuple $(V'_1,\ldots,V'_{2q};\phi'_1,\ldots,\phi'_{2q})$ of $(V_1,\ldots,V_{2q};\phi_1,\ldots,\phi_{2q})$ where we denote $\text{rk}(V'_i) = k'_i$ and $\text{deg}(V'_i) = e_i'$, we have
$$\begin{array}{ccc}
\displaystyle e' < \frac{e}{n-1}\sum_{i=1}^{2q}\frac{k'_i}{k_i} +\sum_{i=1}^{2q}\sigma_i\left(1+k'-(k+1)\frac{k'_i}{k_i}\right) & \mbox{ if } & \phi_1 \in H^0(X,V'_1)\setminus \{0\}\\\\
\displaystyle e' < \frac{e}{n-1}\sum_{i=1}^{2q}\frac{k'_i}{k_i} +\sum_{i=1}^{2q}\sigma_i\left(k'-(k+1)\frac{k'_i}{k_i}\right) & \mbox{ if } & \phi_1 \not\in H^0(X,V'_1)\setminus \{0\}.
\end{array}$$
\end{definition}

\subsection{Pullback diagrams}

The connection between twisted representations of an argyle quiver on the Riemann surface $X$ and $4q$-tuples is captured by the following result:

\begin{theorem}\label{ThmPullback}
For a labelled argyle quiver $Q$ of length $n=2q+1$
\begin{equation}\nonumber
 \bullet_{1,d_1} \longrightarrow \bullet_{r_2,d_2}\longrightarrow\bullet_{1,d_3}\longrightarrow\cdots\longrightarrow\bullet_{r_{n-1},d_{n-1}}\longrightarrow \bullet_{1,d_n}
\end{equation}
there exists a unique $\sigma\in\RR^{2q}$ and $b_i\in\ZZ$ such that the moduli space of representations of Q in the twisted category of holomorphic vector bundles with fixed determinant $P$ is given by the pullback diagram

\begin{equation}\label{pullback}\nonumber
\begin{tikzpicture}
    \node (a) at (0,0){$\calM_{X,L,P}(Q)$};
 \node[right=3cm of a] (b){
$\prod\limits_{i=1,\text{\emph{odd}}}^n \text{\emph{Jac}}^{d_i}(X)$
 };
\node[below = 1.83cm of a](c){ $\calR^\sigma_{X}(k_1,\dots,k_{2q};e_1,\dots,e_{2q})$};
\node[below = 1.3cm of b](d){$\prod\limits_{i=1,\text{\emph{odd}}}^n \text{\emph{Jac}}^{b_i}(X)$};

 \draw[->](a) -- (b)  node[pos=0.5,above]{$g$};
 \draw[->](a) -- (c)  node[pos=0.5,left]{$\pi$};
 \draw[->](c) -- (d)  node[pos=0.5,above]{$h$};
 \draw[->](b) -- (d)  node[pos=0.5,left]{$\pi'$};
\end{tikzpicture}
\end{equation}

with maps described as follows: 
\begin{equation*}\begin{split}
&\pi:(U_1,\ldots,U_n;\phi_1,\ldots,\phi_{n-1})
\\
&\qquad\mapsto (U_1^*U_2L,U_1^*U_3L^2,U_3^*U_4L,U_3^*U_5L^2,\ldots,U^*_{n-2}U_nL^2;\phi_1,\ldots,\phi_{n-1})\\
&g:(U_1,\ldots,U_n;\phi_1,\ldots,\phi_{n-1})\mapsto(U_1,U_3,\ldots,U_n)\\
&h:(V_1,\ldots,V_{2q};\phi_1,\ldots,\phi_{n-1})
\\
&\qquad\mapsto \Bigg(\bigotimes_{i=1, \text{\emph{odd}}}^{2q-1}\det(V_i),\bigotimes_{i=1, \text{\emph{odd}}}^{2q-3}\det(V_i) \otimes\det(V_{2q-1}V^*_{2q}),\ldots
, \bigotimes_{i=1, \text{\emph{odd}}}^{2q-1}\det(V_iV^*_{i+1})\Bigg)\\
&\pi':(U_1,U_3,\ldots,U_n)
\\
&\qquad\mapsto \Big(PL^{\sum_{i=2,\text{\emph{even}}}^{2q}r_i}(U_1^*)^{r_2+1}(U^*_3)^{r_4+1}\ldots U_{2q+1}^*,
\\
&\qquad\qquad PL^{\sum_{i=2,\text{\emph{even}}}^{2q-2}r_i-r_{2q}}(U_1^*)^{r_2+1}\ldots U^*_{2q-1} (U_{2q+1}^*)^{r_{2q-1}+1},\ldots
\\
&\qquad\qquad\ldots ,PL^{-\sum_{i=2,\text{\emph{even}}}^{2q}r_i}U_1^*(U_3^*)^{r_2+1}\ldots (U_{2q+1}^*)^{r_{2q-1}+1}\Big)
\end{split}\end{equation*}

Moreover, the maps $\pi$ and $\pi'$ are finite-to-one covering maps.
\end{theorem}

\begin{proof}
To show that this diagram commutes, we will consider$$h \circ \pi (U_1,\ldots,U_n;\phi_1,\ldots,\phi_{n-1}).$$Recalling $P=U_1\text{det}(U_2)U_3\text{det}(U_4)U_5\ldots U_{2q+1}$, the first term is
\begin{equation}\begin{split}
\bigotimes_{i=1, \text{odd}}^{2q-1}\text{det}(U_i^*U_{i+1}L) &= \bigotimes_{i=1, \text{odd}}^{2q-1}\left((U_i^*)^{r_{i+1}}\text{det}(U_{i+1})L^{r_{i+1}}\right)
\\
&=P(U^*_1)^{r_2+1}(U_3^*)^{r_4+1}\ldots (U_{2q-1}^*)^{r_{2q}+1}U_{2q+1}^*\nonumber
\end{split}\end{equation}
which is exactly the first term of $\pi'\circ g$.  The other terms are similar.  The unique $b_i$ extolled in the statement of the theorem are nothing but the degrees of these line bundles.

To see that  $\pi$ is a $r^{2g}$-fold covering map, write \begin{equation*}
(V_1,\ldots,V_{2q}) = (U_1^*U_2L,U_1^*U_3L^2,U_3^*U_4L,U_3^*U_5L^2,\ldots,U^*_{n-2}U_nL^2)
\end{equation*}
 and then note 
\begin{equation}
\text{det}(V_1) =(U_1^*)^{r_2}L^{r_2}\left(P^*U_1^*U_3^*\text{det}(U_4)^*U_5^*\ldots U_{2q+1}^*\right).\nonumber
\end{equation}
The reasons we can say that $\pi$ is a finite covering are the following: for $i$ odd, $\text{det}V_i = (U_i^*)^{r_{i+1}}\text{det}(U_{i+1})L^{r_{i+1}}$ and so $\text{det}(U_{i+1})^* =(U_i^*)^{r_{i+1}}L^{r_{i+1}}\text{det}(V_i)$. In addition, $\text{det}(V_{i+1}) = U_{i}^*U_{i+2}L^2$ and so $U_{i+2}^* = U_{i-1}^*U_{i}L^2\text{det}(V_{i+1})^*$.  In particular, this tells us that $\text{det}(U_{i+1})^*$ and $U_{i+2}^*$ can be written in terms of $U_i^*$ and some other known quantities.  By doing this for all odd $i$ from $3$ to $2q-1$, we can write $(U_1^*)^{1+r_2+1+\ldots + r_{2q}+1}$ in known terms.  Then, accounting for torsion in the Jacobian, $\pi$ is an $r^{2g}$-fold covering map.  A similar approach shows that $\pi'$ is a finite-to-one covering map.
\\
\\
\qquad Now it remains to show that there exist unique $(\sigma_1,\ldots,\sigma_{2q})$ for which the above holds. We begin by defining the line bundles $U'_j = \phi_{j-1}(U_{j-1})$ and $U''_j = \phi^{-1}_j(U_{j+1})$ for all $j$ even.  For any line subbundle $U'''_j$ of $U_j$ which is \emph{not} equal to either $U'_j$ or $U''_j$, we can define a subrepresentation $$(0,\ldots,U'''_j,\ldots,0;0,\ldots,0)$$ of $$(U_1,\ldots,U_n;\phi_1,\ldots,\phi_{n-1})\in\calM^\alpha_{X,L,P}(Q).$$   It is clear that stability implies $\text{deg}(U'''_j) < \frac{d}{r}$.  Now, such subrepresentations are in one-to-one correspondence with sub-$4q$-ruples $$(0,\ldots,U_{j-1}^*U'''_jL\ldots,0;0,\ldots,0)$$ of $$(V_1,\ldots,V_{2q};\phi_1,\ldots,\phi_{2q})\in\calR^{\sigma}_{X}(k_1,\dots,k_{2q};e_1,\dots,e_{2q}).$$  By definition, such a $4q$-tuple is stable if and only if 
\begin{equation}\begin{split}
e' = \text{deg}(U_{j-1}^*U'''_jL) &< \frac{e}{2q}\sum_{i=1}^{2q}\frac{k'_i}{k_i} +\sum_{i=1}^{2q}\sigma_i\left(k'-(k+1)\frac{k'_i}{k_i}\right)
\\
&=\frac{e}{2qk_{j-1}}+\sum_{i=1}^{2q}\sigma_i\left(1-(k+1)\frac{k'_i}{k_i}\right)
\\
&=\frac{e}{2qr_j}+ \sigma_{j-1}\left(1-\frac{r}{r_j}\right)+\sum_{i=1, i\neq j-1}^{2q}\sigma_i \nonumber
\end{split}\end{equation}
where 
\begin{equation}\begin{split}
e = \sum_{i=1}^{2q+1}d_i+\sum_{i=1, i \text{odd}}^{2q-1}\left((r_{i+1}+2)t-(r_{i+1}+1)d_i\right). \nonumber
\end{split}\end{equation}
Since we also know that 
\begin{equation}
\text{deg}(U_{j-1}^*U'''_jL) = -d_{j-1}+\text{deg}(U'''_j) +t < -d_{j-1}+\frac{d}{r}+t\nonumber
\end{equation}
we see that equivalence of stability in $\calM^\alpha_{X,L,P}(Q)$ and in $\calR^\sigma_{X}(k_1,\dots,k_{2q};e_1,\dots,e_{2q})$ boils down to the equation
\begin{equation}
-d_{j-1}+\frac{d}{r}+t = \frac{e}{2qr_j}+ \sigma_{j-1}\left(1-\frac{r}{r_j}\right)+\sum_{i=1, i\neq j-1}^{2q}\sigma_i \nonumber
\end{equation}
which allows us to deduce 
\begin{equation}\label{sig1}
 \sigma_{j-1}\left(1-\frac{r}{r_j}\right)+\sum_{i=1, i\neq j-1}^{2q}\sigma_i  =-d_{j-1}+\frac{d}{r}+t-\frac{e}{2qr_j}
\end{equation}
for all $j$ even.
\\
\\
\qquad Finally, considering the subrepresentation $(0,\ldots,U'''_j,\ldots,0;0,\ldots,0)$ again, we note that is also in correspondence with sub-$4q$-tuples $$(0,\ldots0,U^*_{j-1}U'''_jL,U_{j-1}^*U_{j+1}L^2,0\ldots,0;0,\ldots,0)$$ of $$ (V_1,\ldots,V_{2q};\phi_1,\ldots,\phi_{2q})\in\calR^{\sigma}_{X}(k_1,\dots,k_{2q};e_1,\dots,e_{2q}),$$ for which the stability condition is
\begin{equation}\begin{split}\nonumber
e' &< \frac{e}{2q}\sum_{i=1}^{2q}\frac{k'_i}{k_i} +\sum_{i=1}^{2q}\sigma_i\left(k'-(k+1)\frac{k'_i}{k_i}\right)
\\
&=\frac{e}{2q}\left(\frac{1}{k_{j-1}}+\frac{1}{k_j}\right)+\sum_{i=1}^{2q}\sigma_i\left(2-(k+1)\frac{k'_i}{k_i}\right)
\\
&=\frac{e}{2q}\left(\frac{1}{r_j}+\frac{1}{r_{j+1}}\right)+ \sigma_{j-1}\left(2-\frac{r}{r_j}\right) + \sigma_{j}\left(2-\frac{r}{r_{j+1}}\right)+\sum_{i=1, i\neq j-1,j}^{2q}2\sigma_i 
\end{split}\end{equation}
where $e$ is as above and $e'$ is
\begin{equation}\begin{split}
e' &=\text{deg}(U_{j-1}^*U'''_jL) + \text{deg}(U_{j-1}^*U_{j+1}L)
\\
&=-2d_{j-1} + d'''_j+d_{j+1}+3t
\\&<-2d_{j-1} + \frac{d}{r}+d_{j+1}+3t.\nonumber
\end{split}\end{equation}
Hence, we set
\begin{equation}
-2d_{j-1} + \frac{d}{r}+d_{j+1}+3t  = \frac{e}{2q}\left(\frac{1}{r_j}+\frac{1}{r_{j+1}}\right)+ \sigma_{j-1}\left(2-\frac{r}{r_j}\right) + \sigma_{j}\left(2-\frac{r}{r_{j+1}}\right)+\sum_{i=1, i\neq j-1,j}^{2q}2\sigma_i \nonumber
\end{equation}
from which we can calculate
\begin{equation}\begin{split}\label{sig2}
&\sigma_{j-1}\left(2-\frac{r}{r_j}\right) + \sigma_{j}\left(2-{r}\right)+\sum_{i=1, i\neq j-1,j}^{2q}2\sigma_i 
\\
&\qquad\qquad =-2d_{j-1} + \frac{d}{r}+d_{j+1}+3t - \frac{e}{2q}\left(\frac{1}{r_j}+\frac{1}{r_{j+1}}\right)
\end{split}\end{equation}
for all $j$ even.

Now we must only show that the system of equations defined by (\ref{sig1}) and (\ref{sig2}) has a unique solution.  The associated $2q\times2q$ matrix is
\begin{equation}
\Sigma_q =
\left(\begin{matrix}
1-\frac{r}{r_2} &1&1&1&\cdots&1 \\
1& 1&1-\frac{r}{r_4}&1&&\vdots\\
\vdots&&&&&\\
2-\frac{r}{r_2}&2-r&2&2&\cdots&\\
2&2&2-\frac{r}{r_4}&2-r&&\\
\vdots&&&&&\\
2&2&2&2&\cdots&2-r
\end{matrix}\right)\nonumber
\end{equation}
which can be transformed to 
\begin{equation}
\Sigma'_q =
\left(\begin{matrix}
1-\frac{r}{r_2}&1&1&\cdots&1\\
1&1-r&1&&\vdots\\
1&1&1-\frac{r}{r_4}&&\\
\vdots&&&\ddots&\\
1&1&1&\cdots&1-r
\end{matrix}\right)\nonumber
\end{equation}
via elementary row operations.  The determinant of $\Sigma'_q$ can be calculated via the matrix determinant lemma, which states that for an invertible $n\times n$ matrix $A$ and column vectors $u$ and $v$, $\text{det}(A+uv^T) = (1+v^TA^{-1}u)\text{det}(A)$. By factoring $\Sigma'_q$ as 
\begin{equation}
\Sigma'_q =
\left(\begin{matrix}
-\frac{r}{r_2}&0&0&\cdots&0\\
0&-r&0&&\vdots\\
0&0&-\frac{r}{r_4}&&\\
\vdots&&&\ddots&\\
0&0&0&\cdots&-r
\end{matrix}\right)
+
\left(\begin{matrix}
1\\
1\\
\vdots\\
1\\
\end{matrix}\right)
\left(\begin{matrix}
1&1&\cdots&1\\
\end{matrix}\right)\nonumber
\end{equation}
we have

\begin{equation}\begin{split}
 (1+v^TA^{-1}u) &= \left(1-\frac{r_2}{r}-\frac{1}{r}-\frac{r_4}{r}-\ldots-\frac{1}{r}\right)
\end{split}\nonumber\end{equation}
as well as

\begin{equation}\begin{split}
\det(A) &=\left(-\frac{r}{r_2}\right)\left(-r\right)\left(-\frac{r}{r_4}\right)\ldots\left(-r\right).\nonumber
\end{split}\end{equation}

From these we can then calculate
\begin{equation}\begin{split}
\text{det}(\Sigma'_q) 
&=\left(r-q-r_2-r_4-\ldots-r_{2q-1}\right)\frac{(-r)^{2q-1}}{r_2r_4\ldots r_{2q-1}}.\nonumber
\end{split}\end{equation}

Since $r=1+(q+r_2+r_4+\ldots+r_{2q-1})$, the determinant is always nonzero, and the proof is complete.
\end{proof}

One way in which to interpret Theorem \ref{ThmPullback} is that the map $h$ generalizes the determinant map of vector bundles to tuples: the determinant of a $4q$-tuple (which contains $2q$ bundles) is a tuple of $q+1$ determinants.  Hence, the fibres of $h$ are the generalization of moduli spaces of bundles of fixed determinant.\\

\noin\emph{Remarks.} The reason for restricting ourselves to argyle quivers in the theorem is that our analysis of $\sigma$ depends explicitly on the fact that every second bundle was a line bundle (in how we defined $U'''_j$).  We do not expect such a clean formulation of the pullback property in the non-argyle case. When the genus of $X$ is $0$, the image of $h$ is just a point, and so there is no useful fibration structure coming from $h$.  However, $\CM_{X,L,P}(Q)$ is still a finite-to-one cover of $\calR^\sigma_X(k_i;e_i)$.  When $g=1$, the Jacobians and the elliptic curve $X$ itself can be identified, and so $\CM_{X,L,P}(Q)$ fibres over a Cartesian product of the elliptic curve with itself some number of times.  In this case, one can view the pullback procedure as expressing the data of a representation of $Q$, which consists of bundles and twisted maps, in terms of simpler data on $X$, namely a tuple of points, after fixing the determinant of the representation (by picking a fibre of $h$) and up to some choice of roots of unity (the map $\pi$).  In some sense, this picture is reminiscent of the spectral viewpoint and the Hitchin fibration for Higgs bundles, which transforms the data of a Higgs bundle on a Riemann surface $X$ to a point on the Jacobian of another Riemann surface, the so-called ``spectral curve'' of the Higgs bundle \cite{NJH:87,BNR:89}.  In the pullback diagram for tuples, we see products of Jacobians rather than a single Jacobian.

We also stress the general utility of the pullback diagram.  A special case of Theorem \ref{ThmPullback} in \cite{MT:94} (cf. also \cite{PBG:94}) is used to obtain an exact geometric identification of the moduli space of stable pairs (a single bundle with a single map) by variation of stability, wherein the stability parameter is initialized at an extreme value and then the desired moduli space is constructed in steps by flips and flops as the parameter crosses certain walls.  In principle, the same procedure can be applied for tuples associated to the more general argyle quivers above but this would involve quite a number of birational transformations.

\section{Quiver bundles on $\PP^1$}\label{SectP1}

In this section, we seek explicit identifications of moduli spaces of twisted representations of argyle quivers when $X$ is $\PP^1$, the most concrete setting.  We begin with the case where the length of the quiver is $n=3$ and work from there.

\subsection{Type $(1,k,1)$ quivers}\label{Sect1k1}
We begin with the quiver$$Q =\bullet_{1,d_1}\longrightarrow \bullet_{k,d_2} \longrightarrow \bullet_{1,d_3}\nonumber$$and put $r=k+2$ and $d=d_1+d_2+d_3$.  A representation of $Q$ is a tuple of the form $(U_1,U_2,U_3;\phi_1,\phi_2)$ in which $U_1\cong\mathcal{O}(d_1)$ and $U_3\cong \calO(d_3)$ since $\mbox{Pic}(\PP^1)\cong\ZZ$. In addition, $U_2$ splits as \begin{equation}\label{u2splitting}\mathcal{O}(a_1)^{\plus s_1}\oplus\cdots\oplus\mathcal{O}(a_m)^{\plus s_m}\nonumber\end{equation} for some $a_i\in\mathbb{Z}$ and some $s_i>0$, where $\sum_{i=1}^m s_ia_i = d_2$ and $k=\sum s_i$. We always sort the $a_i$'s as $a_1 > a_2 > \ldots> a_m$.  With this information in hand, we can rewrite the representation
\begin{equation}
 \begin{tikzcd}
\calO(d_1)\arrow{r}{\Xi} &U_2 \arrow{r}{\Phi} 
&\mathcal{O}(d)
\end{tikzcd}\nonumber
\end{equation}
as

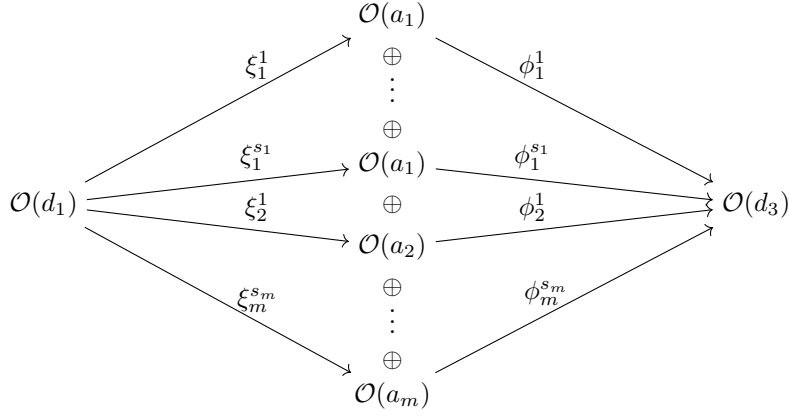
\begin{equation}\label{big1k1quiver}
\begin{tikzpicture}

    \node (a) at (0,0){$\mathcal{O}(a_1)$};
     \node[below=0cm of a] (w){$\oplus$};
    \node[below=0.2cm of a] (b){$\vdots$};
     \node[below=0cm of b] (y){$\oplus$};
 \node[below=0.4cm of b] (c){$\mathcal{O}(a_1)$};
   \node[below=0cm of c] (x){$\oplus$};
 \node[below=0.5cm of c] (d){$\mathcal{O}(a_2)$};
  \node[below=0cm of d] (s){$\oplus$};
 \node[below=0.2cm of d] (e){$\vdots$};
  \node[below=0cm of e] (q){$\oplus$};
 \node[below=0.4cm of e] (f){$\mathcal{O}(a_m)$};
 \node[right=4cm of x] (g){$\mathcal{O}(d_3)$};
  \node[left=3.8cm of x] (h){$\mathcal{O}(d_1)$};

 \draw[->] (a) -- (g) node[pos=0.35,above]{$\phi^1_1$};
 \draw[->] (c) -- (g) node[pos=0.35,above]{$\phi_1^{s_1}$};
 \draw[->] (d) -- (g) node[pos=0.35,above]{$\phi_2^{1}$};
 \draw[->] (f) -- (g) node[pos=0.4,above]{$\phi_m^{s_m}$};
 
  \draw[->] (h) -- (a) node[pos=0.65,above]{$\xi_1^1$};
 \draw[->] (h) -- (c) node[pos=0.65,above]{$\xi_1^{s_1}$};
 \draw[->] (h) -- (d) node[pos=0.65,above]{$\xi_2^{1}$};
 \draw[->] (h) -- (f) node[pos=0.65,above]{$\xi_m^{s_m}$};
    \end{tikzpicture}\nonumber
\end{equation}

In this diagram $\phi_i^j \in H^0 (\mathbb{P}^1,\mathcal{O}(d-a_i+t))$ and $\xi_i^j \in H^0(\PP^1,\calO(d_3-a_i+t))$.  This picture is acted upon by elements of $\text{Aut}(U_1)\times\text{Aut}(U_2)\times\text{Aut}(U_3)$.  From this group, there are degree $0$ maps between each pair of nodes of equal degree, as well as degree $a_i-a_j$ maps from $\calO(a_i)$ to $\calO(a_j)$ for all $i<j$. If we must be very specific, we write $\psi^{ij}_{kl}$ for the map from the $k$-th $a_i$ node to the $l$-th $a_j$ node. Most of the time when considering such maps, it is not important which of the $\calO(a_i)$ nodes we consider, so we simply write $\phi_i$ and $\psi_{ij}$.

Next we will consider which values of $d_1,a_1,\ldots,a_k,d_3$ are allowable under the standard slope-stability conditions.  Since we have already sorted the $a_i$ as $a_1> a_2>\ldots>a_m$, it suffices to impose the following:

\begin{equation}\label{conditions1}\begin{split}
d_3&<\mu_{\text{tot}}\\
\frac{d_3+a_1}{2} &< \mu_{\text{tot}}\\
\vdots\\
\frac{d_3+s_1a_1}{1+s_1} &<\mu_{tot}\\
\frac{d_3+s_1a_1+a_2}{2+s_1} &< \mu_{\text{tot}}\\
\vdots\\
\frac{d_3+\sum_{i=1}^{m} s_ia_i}{k+1} &< \mu_{\text{tot}}\nonumber
\end{split}\end{equation}

Recall that $\phi_i \in H^0(\mathbb{P}^1,\mathcal{O}(d_2-a_i+t))\setminus 0 \cong \mathbb{C}^{d_2-a_i+t+1}\setminus 0$.  Define $i'$ so that $a_{i'+1} < \mu_{\text{tot}} < a_{i'}$ (allowing the cases $i'=1$ or $i'=m$).  This will allow us to say something about the $\phi$ and $\xi$ maps.  We can see that for any $i > i'$, $\calO(a_i)$ can be allowed to be invariant without stability issues, meaning any of the $\phi_i^j$ can be allowed to be zero. On the other hand, for $i>i'$, none of $\xi_i^j$ cannot be allowed to be zero.  If one is zero, then the subbundle consisting of all the nodes except a single $\calO(a_i)$ would be invariant, but this is not stable since $a_{i} < \mu_{\text{tot}}$. In an similar way, for any $i\leq i'$,  $\phi_i^j$ cannot be zero, but $\xi_i^j$ can.  The final restriction to note is that while we allow any of $\xi^1_1,\ldots,\xi_{i'}^{s_{i'}},\phi_{i'+1}^1,\ldots,\phi_m^{s_m}$ to vanish, they cannot all be zero concurrently as that would imply that the representation could be presented as a direct sum of two stable representations.
 
 We will reduce the amount of freedom that some of the $\phi_i$ and $\xi_i$ have by letting them be acted upon by some of the $\psi_{pq}$.  In other words, we construct the moduli variety by performing \emph{reduction in stages}.    We are performing a geometric-invariant-theoretic (GIT) reduction using the $\Phi$-stability condition, but note that we are quotienting by a \emph{non-reductive} group.  In general, an element $\Psi\in\mbox{Aut}(U_i)$ is an invertible matrix-valued polynomial (in the affine parameter $z\in\PP^1$) whose degree $0$ piece is an element of $\mbox{GL}(r_i,\CC)$.  The diagonal terms in particular comprise the usual maximal torus in $\mbox{GL}(r_i,\CC)$.  The off-diagonal terms (which are all zero to one side of the diagonal, by degree considerations) measure the non-reductiveness of the group.  Fortunately, the action of the off-diagonal terms act on the polynomials $\phi_i$ in the representation in a predictable way: they reduce the degree of $\phi_i$ or $\xi_i$ in accordance with the Euclidean algorithm.

To be precise, consider $\psi_{ij}:\calO(a_i)\to\calO(a_j)$ where $a_i\neq a_j$ and $i,j\leq i'$. We send
$\phi_i \mapsto \phi_i + \phi_j\psi_{ij} := \phi_i'$. We know that $\psi_{ij}\in H^0(\mathbb{P}^1,\mathcal{O}(a_j-a_i)) \cong \mathbb{C}^{a_j-a_i+1}$, so one can see that we can use the $a_j-a_i+1$ degrees of freedom of $\psi_{ij}$ to kill off some of the freedom of $\phi_i$.  In particular, the dimension of the space that parametrizes $\phi_i'$ will be $d_2-a_i+t+1 - (a_j-a_i+1) = d_2-a_j+t$. To be more precise, if
\begin{equation}\begin{split}
\phi'_i&= \phi_i + \phi_j\psi_{ij}
\\
&=(A_pz^p+\ldots+A_0) + (B_qz^q+\ldots+B_0)(C_rz^r+\ldots+C_0)\nonumber
\end{split}\end{equation}
then we set $C_r = \frac{-A_p}{B_q}$ so that $C_rB_q = -A_p$, as well as $C_{r-1} = \frac{-1}{B_q}(A_{p-1} + C_rB_{q-1})$ so that $C_rB_{q-1}+C_{r-1}B_q = -A_{p-1}$, etc.  In general, we set
\begin{equation}
C_{r-i} = \frac{-1}{B_q}(A_{p-i} + \sum_{j=0}^{i-1}C_{r-j}B_{m-i+j})\nonumber
\end{equation}
for $i=1,\ldots,r$.  

An additional property of this action is that the size of the automorphism group is not constant; it changes in accordance with divisor equivalences.  This is best explained from the point of view of the spectral correspondence, in which we appeal to the identification of these quiver representations with twisted Higgs bundles.  As previously mentioned, the spectral correspondence \cite{NJH:87,BNR:89} is a bijection between Higgs bundles of fixed generic characteristic polynomial on a curve and line bundles supported on another curve.  This additional curve is called a \emph{spectral curve}, as its points are precisely the spectrum of the Higgs fields on one side of the correspondence.  The spectral curve, $\widetilde X$, is a finite-to-one cover of the original curve ($\PP^1$ in this case), branched over a finite number of points where the characteristic polynomial develops eigenvalues with multiplicity.  The spectral line bundles record the eigenspaces of the Higgs fields.

Most importantly, the spectral correspondence respects isomorphism classes: if two Higgs bundles $(E,\Phi)$ and $(E',\Phi')$ are isomorphic, then their spectral line bundles $L$ and $L'$ are isomorphic, and vice-versa.  If the genus of $\widetilde X$ is $g$, then the Jacobian of $\widetilde X$ is a $g$-dimensional complex torus modelled on the symmetric product $\mbox{S}^g(\widetilde X)$.  It fails to be globally isomorphic to $\mbox{S}^g(\widetilde X)$ because of special divisors.  Specifically, if the degree of the covering map is $r$, then we have an induced surjection $\mbox{S}^g(\widetilde X)\to\PP^r$.  Preimages of points in $\PP^r$ with a repeated coordinate induce extra automorphisms of the corresponding divisors in $\mbox{S}^g(\widetilde X)$.  The quotient of $\mbox{S}^g(\widetilde X)$ by these automorphisms results in $\mbox{Jac}(\widetilde X)$.  The classical example is the Jacobian of the genus $2$ hyperelliptic curve.  The covering map is a degree $2$ map $f:\widetilde X\to\PP^1$.  Its fibres form a $\PP^1$ of linearly-equivalent divisors.  The Jacobian is obtained by blowing down the ``canonical series'' (the preimage of this $\PP^1$ under $\mbox{S}^2(\widetilde X)\to\PP^2$) in $\mbox{S}^2(\widetilde X)$.  In higher genus and for higher degrees of the covering map, these equivalences are more numerous and complicated.

For us, these repeated coordinates in $\PP^r$ correspond to coincidences of invariant zeroes of polynomials in the Higgs fields determined by the representation of the quiver, meaning zeroes of $\phi_i$'s that are preserved by the action of automorphisms.  Suppose that we fix the splitting type $a=(a_1,\dots,a_m;s_1,\dots,s_m)$ of $U_2$ in our quiver.  This is tantamount to adding $2m$ labels to the central node that fix $U_2$.  The resulting moduli space, which we denote $\mathcal{M}_{\mathbb{P}^1,\mathcal{O}(t)}(Q,\mb a)$, keeps track of $\phi_i$ data without any contribution from vector bundle moduli.  We will excise any representations with collisions of invariant zeroes.  We denote the removal of the ``collision manifold'' by a superscript $\Delta$.

\begin{theorem}\label{Thm1k1} Let $Q$ be a quiver of type $(1,k,1)$ and let $\mathbf{a}$ be the splitting type of $U_2$.  The projective closure of $\CM_{\PP^1,\CO(t)}^\Delta(Q,\mathbf{a})$ is
 \begin{equation}\begin{split}
\overline{\mathcal{M}^\Delta_{\mathbb{P}^1,\mathcal{O}(t)}}(Q,\mathbf{a}) &\cong \mathbb{P}^q \times \prod_{j=1}^{i'}\emph{Gr}\Big(s_j,d_3-a_j+t+1-\sum_{k=1}^{j-1}s_{k}(a_k-a_{j}+1)\Big)
\\&\qquad\times \prod_{j=i'+1}^{m}\emph{Gr}\Big(s_j,a_j-d_1+t+1-\sum_{k=j}^{m-1}s_{k}(a_k-a_{j}+1)\Big)\nonumber
\end{split}\end{equation}
where
\begin{equation*}\begin{split}
q=\sum_{j=1}^{i'}s_j(d_3-a_j+t+1) + \sum_{j=i'+1}^{m}s_j(a_j-d_1+t+1)-1
-\sum_{j=1}^{i'}\sum_{k=i'+1}^m s_js_k(a_j-a_k+1).
\end{split}\end{equation*}\end{theorem}

\begin{proof}  We can act on all $\phi_i$ for $i\leq i'$ by all maps $\psi_{ij}$ that go from $\calO(a_i)$ to nodes of higher degree by the Euclidean algorithm, and similarly on all $\xi_i$ for $i'<i$  by maps $\psi_{ij}$ that go to $\calO(a_j)$ from nodes of lower degree.  It is important to note that if the power of $\psi_{ij}$ would reduce the amount of freedom of one of these maps (which are not allowed to be zero by stability) to zero, then the representation is not stable.  Lastly, $\psi_{ij}$ for $j\leq i'<i$ each reduce the freedom of one of  $\xi^1_1,\ldots,\xi_{i'}^{s_{i'}},\phi_{i'+1}^1,\ldots,\phi_m^{s_m}$.  We know that not all these can simultaneously vanish, so they contribute a single projective space to the moduli variety.

Now, after ``using up'' the power of the $\psi_{ij}$ between nodes of different degree and accounting for the data contributed by  $\xi^1_1,\ldots,\xi_{i'}^{s_{i'}},\phi_{i'+1}^1,\ldots,\phi_m^{s_m}$, we can split up and rewrite the remaining information as 
\begin{equation}
\begin{tikzpicture}

 \node (a) at (0,0){$\mathcal{O}$};
    \node[below=0cm of a] (b){$\oplus$};
 \node[below=0.3cm of a] (c){$\vdots$};
 \node[below=0cm of c] (d){$\oplus$};
 \node[below=0cm of d] (e){$\mathcal{O}$};
  \node (f)[right=1.2cm of c]{$\mathcal{O}(d_3-a_1)$};

 \draw[->] (a) -- (f) node[pos=0.35,above]{${\phi_1^1}$};
 \draw[->] (e) -- (f) node[pos=0.35,above]{${\phi^{s_1}_1}$};

\node[right = 0.1 cm of f](g){$\;\;\;\;\cdots$};

\node [right = 4.8cm of a] (h){$\mathcal{O}$};
    \node[below=0cm of h] (i){$\oplus$};
 \node[below=0.3cm of h] (j){$\vdots$};
 \node[below=0cm of j] (k){$\oplus$};
 \node[below=0cm of k] (l){$\mathcal{O}$};
  \node (m)[right=1.2cm of j]{$\mathcal{O}(d_3-a_{i'})$};

 \draw[->] (h) -- (m) node[pos=0.35,above]{${\phi^1_{i'}}$};
 \draw[->] (l) -- (m) node[pos=0.35,above]{${\phi^{s_{i'}}_{i'}}$};

    \end{tikzpicture}\nonumber
    \end{equation}
    
    and 
    \begin{equation}
\begin{tikzpicture}

 \node (a) at (0,0){$\mathcal{O}$};
    \node[below=0cm of a] (b){$\oplus$};
 \node[below=0.3cm of a] (c){$\vdots$};
 \node[below=0cm of c] (d){$\oplus$};
 \node[below=0cm of d] (e){$\mathcal{O}$};
  \node (f)[left=1.2cm of c]{$\mathcal{O}(d_1-a_1)$};

 \draw[->] (f) -- (a) node[pos=0.35,above]{${\xi_{i'+1}^1}$};
 \draw[->] (f) -- (e) node[pos=0.5,above]{${\xi^{s_{i'+1}}_{i'+1}}$};

\node[right = 0.1 cm of c](g){$\;\;\;\;\cdots$};

\node [right = 4.8cm of a] (h){$\mathcal{O}$};
    \node[below=0cm of h] (i){$\oplus$};
 \node[below=0.3cm of h] (j){$\vdots$};
 \node[below=0cm of j] (k){$\oplus$};
 \node[below=0cm of k] (l){$\mathcal{O}$};
  \node (m)[left=1.2cm of j]{$\mathcal{O}(d_1-a_m)$};

 \draw[->] (m) -- (h) node[pos=0.35,above]{${\xi^1_m}$};
 \draw[->] (m) -- (l) node[pos=0.5,above]{${\xi^{s_m}_m}$};

    \end{tikzpicture}\nonumber
    \end{equation}
    
    Write $\Phi_i := {\phi_i^1}'\oplus\ldots\oplus{\phi_i^{s_i}}'$. We claim that the induced map of sections for each of these is, in fact, injective.  If $\tilde\Phi_1:\mathbb{C}^{s_1}\to \mathbb{C}^{d_3-a_1}$ is not injective, then there exists some nontrivial kernel $A$ which is generated by some subbundle $B$ of $\mathcal{O}\oplus\ldots\oplus\mathcal{O}$.  We can say $\text{rank}B < s$ and also note that $B$ must have sections since $A$ is nontrivial. If $\text{rank}B = 1$,  the only degree of $B$ that allows $B$ to have sections is zero, in which case $B$ is destabilizing.  If $\text{rank}B \geq 2$, it is possible that $\text{deg}B \leq -1$  and $B$ can have sections and may not be destabilizing.  However, $B$ must have some subbundle with non-negative degree, which would be destabilizing.  Thus, $\tilde\Phi_1$ is injective, and contributes $\text{Gr}(s_i,d_3-a_1+t+1)$ to the moduli space.  The same argument holds for any $\tilde\Phi_i:\mathbb{C}^{s_i}\to \mathbb{C}^{d_3-a_i+t-\sum_{j=1}^{i-1}s_j(a_j-a_i+1)}$ once noting that the reductions done above can be done in such a way that each $\phi'_i$ induces a map from $\mathbb{C}$ into the subspace $\mathbb{C}^{d_3-a_i+t-\sum_{j=1}^{i-1}s_j(a_j-a_i+1)}$ of $\mathbb{C}^{d_3-a_i+t+1}$, which corresponds to the space of degree $d_3-a+t-2s$ polynomials.  That is, each of the reduced $\phi'_i$ maps into the `same' $\mathbb{C}^{d_3-a_i+t-\sum_{j=1}^{i-1}s_j(a_j-a_i+1)}$.  Moreover, the equality of the moduli spaces of a quiver and its dual allows us to state a similar result for $\Xi_i = {\xi_i^1}'\oplus\ldots\oplus{\xi_i^{s_i}}'$. In particular, it contributes  $\text{Gr}\Big(s_j,a_j-d_1+t+1-\sum_{k=j}^{m-1}s_{k}(a_k-a_{j}+1)\Big)$ to the moduli space. \end{proof}

In the sequel, we reintegrate the collision manifold by identifying it with a twisted $(1,k,1)$ quiver variety for a different splitting type, leading to a stratification of $\mathcal{M}_{\mathbb{P}^1,\mathcal{O}(t)}(Q,\mathbf{a})$ by the algebraic type of $U_2$.

\subsection{General argyle quivers} 
The structure of an argyle quiver allows us to calculate the moduli space as a product of appropriately adjusted $(1,k,1)$ quiver varieties.

\begin{theorem}\label{ThmArgyle} Given a general argyle quiver $Q$ with $\mathbf{a}_i$ the splitting type of $U_i$, the projective completion of the regular part of the moduli space of representations of $Q$ in the category of $\calO(t)$-twisted holomorphic vector bundles over $\PP^1$ is
 \begin{equation}\begin{split}
 \overline{\calM^\Delta_{\PP^1,\calO(t)}}(Q,\mathbf{a}_2,\mathbf{a}_4,\dots,\mathbf{a}_{n-1}) 
&=  \overline{{\calM'}^\Delta_{\PP^1,\calO(t)}}(\bullet_{1,d_1} \longrightarrow \bullet_{r_2,d_2}\longrightarrow\bullet_{1,d_3},\mathbf{a}_2)\times\dots
 \\
 &\qquad\dots\times\overline{{\calM'}^\Delta_{\PP^1,\calO(t)}}(\bullet_{1,d_{n-2}} \longrightarrow \bullet_{r_{n-1},d_{n-1}}\longrightarrow\bullet_{1,d_n},\mathbf{a_{n-1}})
 \nonumber
 \end{split}\end{equation}
 where $$\overline{{\calM'}^\Delta_{\PP^1,\calO(t)}}( \bullet_{1,d_i} \longrightarrow \bullet_{r_{i+1},d_{i+1}}\longrightarrow\bullet_{1,d_{i+2}},\mathbf{a}_{i+1})$$ is the projective completion of the moduli space of the quiver  $$\bullet_{1,d_i} \longrightarrow \bullet_{r_{i+1},d_{i+1}}\longrightarrow\bullet_{1,d_{i+2}}$$ with splitting type of $U_i$ given by $\mathbf{a}_i$, with stability condition induced by $Q$.
\end{theorem}

\begin{proof}
Given a general argyle quiver 
\beqn
Q= \bullet_{1,d_1} \longrightarrow \bullet_{r_2,d_2}\longrightarrow\bullet_{1,d_3}\longrightarrow\cdots\longrightarrow\bullet_{r_{n-1},d_{n-1}}\longrightarrow \bullet_{1,d_n}\nonumber
\eeqn
we can write a representation $(U_1,\dots,U_n;\phi_1,\dots,\phi_{n-1})$ as
\beqn
\begin{tikzpicture}

    \node (a) at (0,0){$\mathcal{O}(a_1)$};
    \node[below=0cm of a] (b){$\oplus$};
 \node[below=0.3cm of a] (c){$\vdots$};
 \node[below=0cm of c] (d){$\oplus$};
 \node[below=0cm of d] (e){$\mathcal{O}(a_{m_a})$};
 
  \node (f)[right=1.2cm of c]{$\mathcal{O}(d_{i+1})$};

 \node (g)[right=3.1cm of a]{$\mathcal{O}(b_1)$};
    \node[below=0cm of g] (h){$\oplus$};
 \node[below=0.3cm of g] (i){$\vdots$};
 \node[below=0cm of i] (l){$\oplus$};
 \node[below=0cm of l] (m){$\mathcal{O}(b_{m_b})$};
 
 \node (n)[right=1.35cm of i]{$\mathcal{O}(d_{i+3})$};

  \node (o)[right=7.5cm of a]{$\mathcal{O}(c_1)$};
    \node[below=0cm of o] (p){$\oplus$};
 \node[below=0.3cm of o] (q){$\vdots$};
 \node[below=0cm of q] (r){$\oplus$};
 \node[below=0cm of r] (s){$\mathcal{O}(c_{m_c})$};

\node[right = 1cm of q]{$\cdots$};
\node[left = 1cm of c]{$\cdots$};

 \draw[->] (a) -- (f) node[pos=0.4,above]{$\zeta_1$};
  \draw[->] (e) -- (f) node[pos=0.4,above]{$\zeta_{m_a}$};
  
   \draw[->] (f) -- (g) node[pos=0.4,above]{$\phi_1$};
  \draw[->] (f) -- (m) node[pos=0.6,above]{$\phi_{m_b}$};
  
     \draw[->] (g) -- (n) node[pos=0.4,above]{$\xi_1$};
  \draw[->] (m) -- (n) node[pos=0.4,above]{$\xi_{m_{b}}$};
  
     \draw[->] (n) -- (o) node[pos=0.4,above]{};
  \draw[->] (n) -- (s) node[pos=0.6,above]{};

    \end{tikzpicture}\nonumber
    \eeqn
    
The conditions on the degrees of the nodes that allow stability are akin to those shown for the $(1,k,1)$ case, although there are many more.  From this picture, it is clear that whether some $\zeta_j$ is allowed to be zero or not, they do not effect the behaviour of the $\phi_i$ in terms of stability, and vice versa.  The same is not true of $\phi_i$ and $\xi_i$, as we have seen.  This suggests that we could consider the moduli space of $Q$ as decomposing as the moduli of the ``diamonds''.  Since the bundles associated to nodes labelled with rank $1$ are fixed, this does not account for any information more than once.  Thus, to calculate $\overline{\mathcal{M}^\Delta_{\PP^1,\calO(t)}}(Q,\mathbf{a}_2,\mathbf{a}_4,\dots,\mathbf{a}_{n-1})$, we only need to calculate  $\overline{\mathcal{M}^\Delta_{\PP^1,\calO(t)}}( \bullet_{1,d_i} \longrightarrow \bullet_{r_{i+1},d_{i+1}}\longrightarrow\bullet_{1,d_{i+2}},\mathbf{a}_{i+1})$ for each of the $(1,k,1)$ blocks, with the following difference: $i'$ is defined so that $a_{i'+1} < \mu_{\text{tot}} < a_{i'}$, where $\mu_{tot}$ is the slope of $Q$, not only the slope of the particular $(1,k,1)$ block.

\end{proof}

\subsection{Stratification of the moduli space by collisions}

In the preceding section, we computed the closure of a single stratum of the $(1,k,1)$ moduli space corresponding to fixing the holomorphic type of the rank $k$ piece and removing collision data.  In this section, we explore examples of how to glue the strata in some low $r$ and low $t$ cases, by realizing one stratum as the ``collision submanifold'' of a more generic stratum. In a sense, we take a finer look at the invariant theory of the representations by indentifying explicit invariants of the isomorphism class that coordinatize the strata.  These invariants take the form of zeroes of certain $\phi_i$'s, regarded as polynomials over $\PP^1$.\\

In the type-change stratification, the largest-dimensional stratum corresponds to representations of generic type, where ``generic'' means precisely the following:

\begin{definition}
Given a bundle $U$ on $\mathbb{P}^1$ of rank $r$ and degree $d$, its \emph{generic splitting} is the decomposition of $U$ as
\begin{equation}
\underbrace{\mathcal{O}(a+1) \oplus \mathcal{O}(a+1)\oplus\ldots\oplus\mathcal{O}(a+1)}_{s} \oplus \underbrace{\mathcal{O}(a)\oplus\ldots\oplus\mathcal{O}(a)}_{r-s}\nonumber
\end{equation}
such that $s(a+1)+(r-s)a=d$.
\end{definition}

The bundle $U$ admits other infinitely many ``less generic'' splitting types that are related to the one above by adding $1$ to the degree of a summand and simultaneously removing $1$ from the degree of another summand.  As per usual, it is stability that caps the number of splitting types that appear in the moduli space.

Consider the general $(1,k,1)$ case. For a representation with $U_2$ of type $$(a_1,\dots,a_1;\dots;a_m,\dots,a_m),$$ we have
\begin{equation}
\Phi=
\left(\begin{matrix}
0&0& \cdots &&&&\cdots& 0  \\
 \xi^1_1&0&  &&&&&\vdots \\
 \vdots &\vdots&\ddots&&&&&\\
 \xi^{s_1}_1&&&&&&&\\
 \xi^1_2&&&&&&&\\
  \vdots &&&&&&&\\
  \xi_m^{s_m}&0&&&&&&\vdots\\
0&\phi^1_1&\cdots&\phi_1^{s_1}& \phi^1_2&\cdots&\phi_m^{s_m}& 0\\\nonumber
\end{matrix}\right).
\end{equation}
By observing $\Psi^{-1}\Phi\Psi$, we see that $\phi_i^j$ will have an invariant zero if and only if $\phi_1^1,\dots,\phi_1^{s_1},\dots,\phi_i^1,\dots,\phi_i^{s_i}$ have a common zero.  As well, $\xi^j_i$ will have an invariant zero if and only if $\xi_j^1,\dots,\xi_j^{s_j},\dots,\xi_m^1,\dots,\xi_m^{s_m}$ have a common zero.  For $1\leq i \leq j\leq k$, we would like to construct a way to map a representation with $U_2$ of type $(a_1,\dots,a_1;\dots;a_m,\dots,a_m)$ to a representation with of the same type except that a term $a_i$ has been replaced with $b_i+1$ and a $a_j$ has been replaced with $a_j-1$.  In view of the above description of the invariant zeroes, it is possible to construct a meromorphic automorphism $\Theta$ that will make the above transformation exactly in the case that $\phi_1^1,\ldots,\phi_i^{s_i},\xi_{j}^1,\ldots\xi_m^{s_m}$ all share a zero, which is precisely when the automorphism has determinant equal to $1$ (as opposed to having a determinant which is a meromorphic section of $\CO$).

We pose the following algorithm that controls how the holomorphic type of $U_2$ changes due to a collision of invariant zeroes:

\textbf{The Type-Change Algorithm:}
Begin with an empty set $\calS$. Given a splitting $S_0 =(a_1,\dots,a_1;\dots;a_m,\dots,a_m)$ of $U_2$ (where $a_i$ appears $s_i$ times), add $S_0$ to $\calS$ then apply the following:

Given $S_p =(b_1,\dots,b_1;\dots;b_m,\dots,b_m)$, choose integers $i,j$ such that $1\leq i \leq j\leq k$ then construct the sequence $S_{p+1}$, which is identical to $S_p$, save for that a term $b_i$ has been replaced with $b_i+1$ and a $b_j$ has been replaced with $b_j-1$.  If this $S_{p+1}$ is not in $\calS$ and the corresponding representation type is stable, then we glue in the moduli space of representations corresponding to type $S_{p+1}$ in place of the collision locus of type $S_p$. Then, add $S_{p+1}$ to $\calS$ and restart this procedure with $S_{p+1}$.  If $S_{p+1}$ is unstable, then add it to $\calS$, and if $p>0$, apply the procedure to $S_{p-1}$.  If $p=0$, terminate.

The moduli space of representations of the quiver$$Q = \bullet_{1,d_1} \longrightarrow \bullet_{r_2,d_2}\longrightarrow\bullet_{1,d_3}$$in the category of $\calO(t)$-twisted holomorphic vector bundles over $\PP^1$ can then be geometrically realized as the moduli space corresponding to the generic splitting, subject to the type-change algorithm.

As an example, consider the quiver $Q=\bullet_{1,2}\longrightarrow\bullet_{2,-1}\longrightarrow\bullet_{1,-2}$ with $\mb a = (0;-1)$ and $t=5$.  A representation of $Q$ looks like
\begin{equation}
\begin{tikzpicture}

    \node (a) at (0,0){$\mathcal{O}$};
     \node[below=0.4cm of a] (w){$\oplus$};
    \node[below=1.3cm of a] (b){$\calO(-1)$};
 \node[right=1.8cm of w] (g){$\mathcal{O}(-2)$};
  \node[left=1.8cm of w] (h){$\mathcal{O}(2)$};

 \draw[->] (a) -- (g) node[pos=0.5,above]{$\phi_1$};
  \draw[->] (b) -- (g) node[pos=0.35,above]{$\phi_2$};
  \draw[->] (h) -- (a) node[pos=0.5,above]{$\xi_1$};
  \draw[->] (h) -- (b) node[pos=0.65,above]{$\xi_2$};
    \end{tikzpicture}\nonumber
\end{equation}
Here, $\xi_1\in H^0(\PP^1,\calO(3))$, $\xi_2\in H^0(\PP^1,\calO(2))$, $\phi_1\in H^0(\PP^1,\calO(3))$ and $\phi_2\in H^0(\PP^1,\calO(4))$.  By stability, $\xi_2$ and $\phi_1$ are not allowed to vanish, and they contribute $\PP^2$ and $\PP^3$ to the moduli space respectively.  Either of $\xi_1$ or $\phi_2$ can be zero, but they cannot vanish concurrently. The automorphism $\psi_{21}:\calO(-1)\to\calO$, $\psi_{21}\in H^0(\PP^1,\calO(1))$ acts on either of $\xi_1$ or $\phi_2$, reducing the amount of freedom by $2$, and so $(\xi_1,\phi_2)$ contributes $\PP^6$.  Hence
\begin{equation}
\overline{\mathcal{M}^\Delta_{\PP^1,\calO(5)}}( Q,\mathbf{a}) = \PP^2\times\PP^3\times\PP^6.\nonumber
\end{equation}
The only other splitting type of $U_2$ which corresponds to a stable representation of $Q$ is $\mathbf{b}=(1,-2)$.  Such a representation looks like
\begin{equation}
\begin{tikzpicture}

    \node (a) at (0,0){$\mathcal{O}(1)$};
     \node[below=0.4cm of a] (w){$\oplus$};
    \node[below=1.3cm of a] (b){$\calO(-2)$};
 \node[right=1.8cm of w] (g){$\mathcal{O}(-2)$};
  \node[left=1.8cm of w] (h){$\mathcal{O}(2)$};

 \draw[->] (a) -- (g) node[pos=0.5,above]{$\phi'_1$};
  \draw[->] (b) -- (g) node[pos=0.35,above]{$\phi'_2$};
  \draw[->] (h) -- (a) node[pos=0.5,above]{$\xi'_1$};
  \draw[->] (h) -- (b) node[pos=0.65,above]{$\xi'_2$};
    \end{tikzpicture}\nonumber
\end{equation}
In a way completely analagous to the above, we have
\begin{equation}
\overline{\mathcal{M}^\Delta_{\PP^1,\calO(5)}}( Q,\mathbf{b}) = \PP^1\times\PP^2\times\PP^6.\nonumber
\end{equation}
We can identify this space with one of the collision manifolds of $\mathcal{M}_{\PP^1,\calO(5)}(Q,\mathbf{a})$; in particular, when $\xi_2$ and $\phi_1$ share a zero $z'$, we can construct the following meromorphic automorphism
\begin{equation}
\Theta = 
 \left(\begin{matrix}
 1&0&0&0\\
0&\frac{1}{z-z'} &0&0 \\
0&0& z-z'&0\\
0&0&0&1\\
\end{matrix}\right)\nonumber
\end{equation}
that acts by conjugation to take a representation with $U_2$ of type $\mathbf{a}$ to a representation with $U_2$ of type $\mathbf{b}$. This amounts to a change of basis of the Higgs field.   Moreover, in this case we can make a fairly explicit identification of the full moduli space $\mathcal{M}_{\PP^1,\calO(5)}(Q)$: it is $ \PP^2\times\PP^3\times\PP^6$ blown down to $\PP^1\times\PP^2\times\PP^6$ along the collision locus of $\xi_2$ and $\phi_1$, which lies in $ \PP^2\times\PP^3$.

For a slightly trickier example, consider 
$Q=\bullet_{1,2}\longrightarrow\bullet_{2,0}\longrightarrow\bullet_{1,-3}$ with $\mb a = (0,0)$ and $t=6$.  A representation of $Q$ looks like
\begin{equation}
\begin{tikzpicture}

    \node (a) at (0,0){$\mathcal{O}$};
     \node[below=0.4cm of a] (w){$\oplus$};
    \node[below=1.3cm of a] (b){$\calO$};
 \node[right=1.8cm of w] (g){$\mathcal{O}(-3)$};
  \node[left=1.8cm of w] (h){$\mathcal{O}(2)$};

 \draw[->] (a) -- (g) node[pos=0.5,above]{$\phi_1$};
  \draw[->] (b) -- (g) node[pos=0.35,above]{$\phi_2$};
  \draw[->] (h) -- (a) node[pos=0.5,above]{$\xi_1$};
  \draw[->] (h) -- (b) node[pos=0.65,above]{$\xi_2$};
    \end{tikzpicture}\nonumber
\end{equation}
Here, $\xi_1,\xi_2\in H^0(\PP^1,\calO(4))$ and $\phi_1,\phi_2 \in H^0(\PP^1,\calO(3))$.  By stability, neither $\phi_1$ nor $\phi_2$ can be zero, and $\xi_1$ and $\xi_2$ cannot be zero concurrently.  Hence,
\begin{equation}
\overline{\mathcal{M}^\Delta_{\PP^1,\calO(6)}}( Q,\mathbf{a}) = \PP^9\times\text{Gr}(2,4).\nonumber
\end{equation}

Once again there is only one other splitting type of $U_2$ which corresponds to a stable representation, and in this case it is $\mathbf{b} = (1;-1)$.  Such a representation has moduli space calculated in the same way as the first example:
\begin{equation}
\overline{\mathcal{M}^\Delta_{\PP^1,\calO(6)}}( Q,\mathbf{b}) = \PP^2\times\PP^3\times\PP^9.\nonumber
\end{equation}
How this space fits into $\PP^9\times\text{Gr}(2,4)$ is not immediately clear.  This is due to the fact that these two representations have different ``stability types'':  the maps that are allowed to be zero and those that are intertwined with each other are different in each of the representation types.  In the generic stratum, neither $\phi_1$ and $\phi_2$ can be zero while $\xi_1$ and $\xi_2$ cannot be simultaneously be zero.  In the less generic stratum, neither $\phi_1$ nor $\xi_2$ can be zero while $\phi_2$ and $\xi_1$ form an analogous pair.  The change in stability in crossing from one stratum to the other is reminiscent of a conifold transition in reductive GIT, but where the dimension need not be the same on both sides of the transition.

Finally, we consider an argyle quiver with two $(1,k,1)$ blocks.  Let 
\begin{equation}
Q=\bullet_{1,0}\longrightarrow\bullet_{2,0}\longrightarrow\bullet_{1,3}\longrightarrow\bullet_{3,-2}\longrightarrow\bullet_{1,-2}\nonumber
\end{equation}

and $t=5$.  The generic splittings $(\mathbf{a}_2,\mathbf{a}_4)$ are $\mathbf{a}_2 = (0,0)$ and $\mathbf{a}_4 = (0;-1,-1)$.  A representation with these splittings looks like
\begin{equation}
\begin{tikzpicture}

    \node (a) at (0,0){$\mathcal{O}$};
     \node[below=0.4cm of a] (w){$\oplus$};
    \node[below=1.3cm of a] (b){$\calO(-1)$};
    \node[below = 0.4cm of b](x){$\oplus$};
    \node[below=1.3cm of b](c){$\calO(-1)$};
 \node[right=1.8cm of b] (g){$\mathcal{O}(-2)$};
  \node[left=1.8cm of b] (h){$\mathcal{O}(3)$};
  
  \node[left=1.8cm of h](y){$\oplus$};
  \node[left=1.8cm of y](z){$\calO$};
  
  \node[above=0.7cm of y](d){$\calO$};
   \node[below=0.7cm of y](e){$\calO$};

 \draw[->] (a) -- (g) node[pos=0.5,above]{$\phi_1$};
  \draw[->] (b) -- (g) node[pos=0.5,above]{$\phi_2$};
   \draw[->] (c) -- (g) node[pos=0.4,above]{$\phi_3$};
  \draw[->] (h) -- (a) node[pos=0.5,above]{$\xi_1$};
  \draw[->] (h) -- (b) node[pos=0.5,above]{$\xi_2$};
    \draw[->] (h) -- (c) node[pos=0.6,above]{$\xi_3$};
    
      \draw[->] (d) -- (h) node[pos=0.5,above]{$\zeta_1$};
  \draw[->] (e) -- (h) node[pos=0.5,above]{$\zeta_2$};
     \draw[->] (z) -- (d) node[pos=0.5,above]{$\eta_1$};
  \draw[->] (z) -- (e) node[pos=0.5,above]{$\eta_2$};
    \end{tikzpicture}\nonumber
\end{equation}

Note that the left block is certainly not a stable representation of the quiver $\bullet_{1,0}\longrightarrow\bullet_{2,0}\longrightarrow\bullet_{1,3}$, but with stability condition induced by $Q$, we can calculate 
\begin{equation}
\overline{{\calM'}^\Delta_{\PP^1,\calO(5)}}(\bullet_{1,0}\longrightarrow\bullet_{2,0}\longrightarrow\bullet_{1,3},\mathbf{a}_2) = \PP^{11}\times\text{Gr}(2,9).\nonumber
\end{equation}
Similarly, 
\begin{equation}
\overline{{\calM'}^\Delta_{\PP^1,\calO(5)}}(\bullet_{1,3}\longrightarrow\bullet_{3,-2}\longrightarrow\bullet_{1,-2},\mathbf{a}_4) = \{\bullet\}\times \PP^3\times\PP^8\nonumber
\end{equation}

Here, $t$ is small enough and the quiver labelling is such that none of the other possible splittings of $U_2$ or $U_4$ correspond to stable representations.  Thus we can actually say
\begin{equation}
\calM_{\PP^1,\calO(5)}(Q) =  \PP^3\times\PP^8\times \PP^{11}\times\text{Gr}(2,9)\nonumber
\end{equation}

If we consider $t=6$ with this same quiver, we observe a stratification which is more difficult to categorize. The splitting $\mb{a}_2 = (0,0)$ is still the only stable type for $U_2$, but for $U_4$ we also have $\mb{b}_4 = (1;-1;-2)$ and $\mb{c}_4 = (0,0;-2)$ corresponding to stable representations.  We calculate
\begin{equation}
\overline{{\calM'}^\Delta_{\PP^1,\calO(6)}}(Q,\mb{a}_2,\mb{a}_4) = \PP^{13}\times\text{Gr}(2,10)\times\PP^4\times\PP^{11}\times\text{Gr}(2,3)\nonumber
\end{equation}
\begin{equation}
\overline{{\calM'}^\Delta_{\PP^1,\calO(6)}}(Q,\mb{a}_2,\mb{b}_4) = \PP^{13}\times\text{Gr}(2,10)\times\PP^1\times\PP^3\times\PP^{10}\nonumber
\end{equation}
and 
\begin{equation}
\overline{{\calM'}^\Delta_{\PP^1,\calO(6)}}(Q,\mb{a}_2,\mb{c}_4) = \PP^{13}\times\text{Gr}(2,10)\times\PP^1\times\PP^8\times\text{Gr}(2,5).\nonumber
\end{equation}

It is unclear how to glue these into the collision loci of $\overline{{\calM'}^\Delta_{\PP^1,\calO(6)}}(Q,\mb{a}_2,\mb{a}_4)$.  This is partially due to the conifold-like transition mentioned earlier and also because $\overline{{\calM'}^\Delta_{\PP^1,\calO(6)}}(Q,\mb{a}_2,\mb{b}_4)$  can be viewed as lying in a collision locus of $\overline{{\calM'}^\Delta_{\PP^1,\calO(6)}}(Q,\mb{a}_2,\mb{c}_4)$, but from the point of view of collisions in  $\overline{{\calM'}^\Delta_{\PP^1,\calO(6)}}(Q,\mb{a}_2,\mb{a}_4)$, $\overline{{\calM'}^\Delta_{\PP^1,\calO(6)}}(Q,\mb{a}_2,\mb{c}_4)$ is a special case of  $\overline{{\calM'}^\Delta_{\PP^1,\calO(6)}}(Q,\mb{a}_2,\mb{b}_4)$.

\section{Applications to twisted Higgs bundles}\label{SectHiggs}

The primary application of twisted quiver representations in a category of bundles is to the topology of Higgs bundle moduli spaces.  Rational Betti numbers have been calculated in a number of cases using the localization to fixed points of the $\CC^\times$ action.  For ordinary Higgs bundles of rank $r=2$ on a Riemann surface of genus $g\geq2$, these were calculated by Hitchin \cite{NJH:86}.  The rank $3$ and rank $4$ cases were computed in \cite{PBG:94} and \cite{GHS:14}, respectively.  In the parabolic Higgs setting on punctured Riemann surfaces, rational Poincar\'e series in low rank were computed in \cite{BY:96} and \cite{GGM:07}.  All of these calculations are largely Morse-theoretic, although \cite{GHS:14} uses moduli stacks and motivic zeta functions.	

In this context, the natural application of our results in the preceding sections (which concern representations over the projective line) is to twisted Higgs bundles at genus $0$.  In this particular setting, there are now general results on Donaldson-Thomas invariants due to Mogovogy in \cite{SM:16}, obtained by plethystic counting techniques, from which the Betti numbers can be extracted.  In comparison, the $\CC^\times$-localization tends to becomes unmanageable outside of low rank due to the number of types of fixed points.  That being said, bearing with it can reap rewards such as information on the stratification of the moduli space, as organized by the Morse flow, as well as about the structure of the cohomology ring (not to mention an abundance of finer information, such as Verlinde formulae \cite{AGP:16,DHL:16}, although this requires a much deeper analysis of the fixed-point geometry).

We denote by $\CH_t(r,d)$ the moduli space of stable twisted Higgs bundles of rank $r$ and degree $d$ on $\PP^1$, where $\gcd(r,d)=1$ (the coprime condition eliminates objects that are semistable but not stable).  In each stable Higgs bundle $(E,\Phi)$, the map $\Phi$ is an $\CO_{\PP^1}$-linear map from $E$ to $E\tensor\CO(t)$, where $t$ is a fixed positive integer. The dimension over $\CC$ of the moduli space is $tr^2+1$ \cite{NN:91}. As noted earlier, this space comes equipped with a linear algebraic action of $\CC^\times$ that sends $(E,\Phi)$ to $(E,\lambda\Phi)$.   Each fixed point of this action is a representation of the quiver $A_n$, for some $n$ with $1\leq n\leq r$, and with a labelling by pairs of integers $r_i,d_i$ in which $\sum r_i=r$ and $\sum d_i=d$ and $r_i>0$ \cite{PBG:95,SR:17}.  When $r>1$, there are no fixed points with length $n=1$, as these correspond to stable Higgs bundles with the zero Higgs field which are simply stable bundles on $\PP^1$, of which there are none other than line bundles.

The action induces a localization of cohomology to the fixed-point locus, and so the Poincar\'e series of $\CH_t(r,d)$ is the weighted sum of the Poincar\'e series of the connected components of the fixed-point set \cite{NJH:86}:\beqn \CP(r,d;x) & = & \sum_\CN x^{\beta(\CN)}\CP(\CN;x),\nonumber\eeqn\noin 
where $\CN$ denotes a connected component of the fixed point set; $\CP(\CN;x)$, the Poincar\'e series of $\CN$; and $\beta(\CN)$, the Morse index of any point in $\CN$.   The indices can be computed algebraically as dimensions of weight spaces or by using differential topology, namely Morse-Bott theory.  The former invokes the GIT interpretation of the moduli space while the latter takes a symplectic point of view.  We use the latter here, meaning that we restrict to the Hamiltonian action of a copy of $S^1\subset\CC^\times$, namely $(E,\Phi)\mapsto(E,\exp(i\theta)\Phi)$.  The procedure for computing the weights for this action on $\CM_t(r,d)$ is spelled out in \S2 of \cite{SR:17} (it is twice the sum of the numbers $\beta(E)$ and $\beta(\Phi)$ appearing there).

The initial case of interest is $r=2$ with any odd $d$ and any $t>0$.  The dimension of $\CH_t(2,d)$ is $4t+1$. There is a single quiver that controls the fixed points: $A_2$--- with nodes labelled $1,a$ and $1,d-a$, respectively.  This is an argyle quiver of type $(1,1)$, for which the moduli space is relatively simple to compute.  For any $a,d,t$, the moduli space is just $\PP^{-2a+d+t}$.   Note that there is no collision or type-change behaviour in this case, as both nodes correspond to line bundles and so $a$ and $d-a$ fix the bundles up to isomorphism.

These components of the fixed-point locus are indexed by $a$ and the admissible values of $a$ are determined by stability.  If $a$ is too large and positive, then the only morphism between the nodes will be the zero map and a copy of $\CO(a)$ will be invariant, with slope larger than $d/2$.  If $a$ is too negative, the copy of $\CO(d-a)$ will be destabilizing.  It is possible to enumerate the labelled quivers directly.  For instance, for $d=-1$, we have $\ds\left\lfloor\frac{t+1}{2}\right\rfloor$ integers $a$ such that $\CO(a)\stackrel{\phi}{\to}\CO(d-a)\to0$ is stable:
\beqn\CO\rightarrow\CO(-1)\rightarrow0\nonumber\\\CO(1)\rightarrow\CO(-2)\rightarrow0\nonumber\\\vdots\;\;\;\;\;\nonumber\\\CO(-1+\lfloor(t+1)/2\rfloor)\rightarrow\CO(-\lfloor(t+1)/2\rfloor)\rightarrow0\nonumber\eeqn

For any other odd $d$, the list will have the same number of entries, but with degrees that have been shifted appropriately.  Using the Betti numbers of $\PP^{-2a+d+t}$ for each admissible $a$, the corresponding Morse index from \cite{SR:17}, and the localization formula, we arrive at: 

\begin{theorem} For any odd $d$ and any $t>0$, we have\beqn\CP(\CH_t(2,d),x) & = &\ds\sum_{k=0}^{t-1}\left(\frac{2k+4-[(2k)\,\mbox{\emph{mod}}\,4]}{4}\right)x^{2k}.\nonumber\eeqn\end{theorem}

The even Betti numbers are $1,1,2,2,3,3,4,4,\dots$ up to $(t-1)/2,(t+1)/2$ if $t$ is odd or $t/2,t/2$ if $t$ is even.  From a combinatorial point of view, these count partitions of even integers into unordered combinations of the numbers $2$ and $4$, i.e. the ``change-making problem''.  To emphasize this, one can rewrite the series as$$\CP(\CH_t(2,d),x)=\frac{1}{(1-x^2)(1-x^4)} - \left\{\frac{(\lfloor t/2 \rfloor + 1)x^{2t}}{1-x^2} + \frac{x^{4 \lfloor t/2 \rfloor + 4}}{(1-x^2)(1-x^4)}\right\},$$which displays more of the structure regarding the generators and relations in the cohomology ring  (These results in the $t=2$ or ``co-Higgs bundle'' case were found in \cite{SR:13}).

In the rank $3$ case, the quiver types are now $(1,1,1)$, $(1,2)$, and $(2,1)$, all of which are argyle.  In this case, we must contend with collisions, which makes writing down a general Poincar\'e series cumbersome.  We provide two examples, one without type change and one with.

For the first example, we consider $\CH_2(3,-1)$, seen also in \cite{SR:13}.  The complex dimension of the moduli space is $19$ in this case.  As with $r=2$, the fixed-point set consists entirely of representations of argyle quivers, the types being $(1,1,1)$, $(2,1)$, and $(1,2)$. Stability rapidly eliminates any of type $(1,2)$.  For type $(1,1,1)$, there are three degree labellings that produce stable representations:$$1,0,-2;\;\;1,-1,-1;\;\mbox{ and }0,0,-1,$$which have Morse indices of $6,4$, and $2$ respectively. Again, there are no type-changing collisions possible because the bundles are line bundles and are therefore fixed up to isomorphism by these degree labellings.  By Theorem \ref{Thm1k1}, the associated quiver varieties are$$\PP^{-1+0+2}\times\PP^{-0-2+2},\;\PP^{-1-1+2}\times\PP^{1-1+2},\;\mbox{ and }\PP^{-0+0+2}\times\PP^{0-1+2},$$respectively.  For type $(2,1)$, there is a single degree labelling that admits stable representations: $0,-1$, which has Morse index $0$ (and so we are at the ``bottom'' of the moduli space).  We can deduce from the arguments leading to Theorem \ref{Thm1k1} that the associated quiver variety is just a point.  More directly, the representation $\phi:\CO\plus\CO\rightarrow\CO(-1)\tensor\CO(2)$ is stable if and only if it is onto, in which case the induced map $\widetilde\phi$ between spaces of global sections must have full rank.  Acting on this copy of $\mb{GL}(2,\CC)$ on the right by automorphisms of $\CO\plus\CO$ leaves nothing save for the identity.  Weaving together this information with the localization formula, we obtain$$\CP(\CH_2(3,-1),x)=1+x^2+3x^4+4x^6+3x^8.$$ As with the $r=2$ case, the top degree is decidedly less than the actual dimension of the moduli space.  This is due to  the contribution to the moduli space of the Hitchin base; the space of possible coefficients of the characteristic polynomial of $\Phi$, which itself is topologically trivial.  The moduli space itself deformation retracts onto the central fibre over the base.

Finally, we consider $\CH_6(3,-1)$.  The basic types are the same ($(1,1,1)$, $(2,1)$, and $(1,2)$) but type-change phenomena occur. Here, the complex dimension of the moduli space is $55$. For type $(2,1)$, the labellings $(0,-1)$, $(1,-2)$, and $(2,-3)$ produce stable representations with Morse indices $0,4$, and $12$ respectively. The variety corresponding to the labelling $(0,-1)$ is $\text{Gr}(2,6)$ and that corresponding to $(1,-2)$ is $\PP^3\times\PP^2$. Each of these labellings has only one splitting of the left node that corresponds to stable representations.  The same is not true of the labelling $(2,-3)$, where we contend with type-change phenomena. We have both
\begin{equation}
\begin{tikzpicture}

    \node (a) at (0,0){$\mathcal{O}(1)$};
     \node[below=0.4cm of a] (b){$\oplus$};
    \node[below=1.3cm of a] (c){$\calO(1)$};
 \node[right=1.5cm of b] (d){$\mathcal{O}(-3)$};

 \draw[->] (a) -- (d) node[pos=0.5,above]{$\phi_1$};
  \draw[->] (c) -- (d) node[pos=0.5,above]{$\phi_2$};
  
   \node[right=0.7cm of d] (e){and};
   
   \node[right=5 of a] (f){$\mathcal{O}(2)$};
     \node[below=0.4cm of f] (g){$\oplus$};
    \node[below=1.3cm of f] (h){$\calO$};
 \node[right=1.5cm of g] (i){$\mathcal{O}(-3)$};

 \draw[->] (f) -- (i) node[pos=0.5,above]{$\phi'_1$};
  \draw[->] (h) -- (i) node[pos=0.45,above]{$\phi'_2$};
   
    \end{tikzpicture}\nonumber
\end{equation}
  The quiver variety of the first is $\text{Gr}(2,3)\cong \PP^2$ and the quiver variety of the second is $\PP^1$.  The locus of  $\PP^2$ where $\phi_1$ and $\phi_2$ share a zero is a copy of $\PP^1$.  We remove this and paste in the second variety, which is just $\PP^1$ again.  So in this case we have that the moduli space correponding to this labelling of a type $(2,1)$ quiver is $\PP^2$.  In addition, we have two stable labellings of the $(1,2)$-type quiver, $(1,-2)$ and $(2,-3)$ with respective Morse indices $4$ and $10$.  The labelling $(1,-2)$ has associated quiver variety $\text{Gr}(2,4)$, and $(2,-3)$ has $\PP^2\times\PP^1$.  Finally, we have the following allowed labellings for the $(1,1,1)$ quiver type:
  \begin{equation*}\begin{split}&0,0,-1;\;\;0,1,-2;\;\;1,-1,-1;\;\;0,2,-3;\;\;1,0,-2;\;\;2,-2,-1;\;\;1,1,-3;
  \\
  &2,-1,-2;\;\;3,-3,-1;\;\;1,2,-4;\;\;2,0,-3;\;\;3,-2,-2;\;\;2,1,-4;\;\;3,-1,-3;;
  \\
  &3,0,-4;\;\;4,-2,-3;\;\;3,1,-5;\;\;4,-1,-4;\;\;4,0,-5;\;\;5,-1,-5;\mbox{ and }5,0,-6.
  \end{split}\end{equation*}
  These have Morse indices \begin{equation*}10,12,12,14,14,14,16,16,16,18,18,18,20,20,22,22,24,24,26,28\mbox{ and }30\end{equation*} respectively, and associated quiver varieties
    \begin{equation*}\begin{split}&\PP^6\times\PP^5,\;\PP^7\times\PP^3,\;\;\PP^4\times\PP^6,\;\;\PP^8\times\PP^1,\;\;\PP^5\times\PP^4,\;\;\PP^2\times\PP^7,\;\;\PP^6\times\PP^2,
  \\
  &\PP^3\times\PP^5,\;\;\PP^8,\;\;\PP^7,\;\;\PP^4\times\PP^3,\;\;\PP^1\times\PP^6,\;\;\PP^5\times\PP^5,\;\;\PP^2\times\PP^4,
  \\
  &\PP^3\times\PP^2,\;\;\PP^5;\;\;\PP^4,\;\;\PP^1\times\PP^3,\;\;\PP^2\times\PP^1,\;\;\PP^2,\mbox{ and }\PP^1.
  \end{split}\end{equation*}
  We can bring all of this together to calculate
  \begin{equation*}\begin{split}
  \CP(\CH_6(3,-1),x) &= 1+x^2+3x^4+4x^6+7x^8+9x^{10}+14x^{12}+17x^{14}+24x^{16}+29x^{18}
 \\
 &\qquad+38x^{20}5x^{22}+49x^{24}+49x^{26}+45x^{28}+36x^{30}+21x^{32}.
 \end{split} \end{equation*}
After $r=3$, $\CH_t(r,d)$ will always contain topological contributions from at least one $A$-type quiver of non-argyle type.  For instance, $r=4$ contains a $(2,2)$ quiver variety, which was for some time the obstruction to computing Betti numbers for ordinary Higgs bundles in higher genus before \cite{GHS:14}.  On $\PP^1$, the $(2,2)$ quiver is not so formidable and, with some effort, one can find
\beqn\begin{split}
\CP(\CH_2(4,d),x) &= 1+x^2+3x^4+5x^6+9x^8+13x^{10}+18x^{12}
\\
&\qquad+22x^{14}+20x^{16}+10x^{18},
\nonumber\end{split}\eeqn
 for instance, where $d$ is any integer coprime to $4$.  We remark finally that all of the above calculations agree with the conjectural Poincar\'e series for these moduli spaces arising from the ADHM recursion formula \cite{SM:12}. 

\bibliographystyle{acm} 
\bibliography{ArgyleQuivers}

\begin{thebibliography}{10}

\bibitem{AG:01}
{\sc {\'A}lvarez-C{\'o}nsul, L., and Garc{\'i}a-Prada, O.}
\newblock Dimensional reduction, {${\rm SL}(2,\Bbb C)$}-equivariant bundles and
  stable holomorphic chains.
\newblock {\em Internat. J. Math. 12}, 2 (2001), 159--201.

\bibitem{AG:03}
{\sc {\'A}lvarez-C{\'o}nsul, L., and Garc{\'i}a-Prada, O.}
\newblock Dimensional reduction and quiver bundles.
\newblock {\em J. Reine Angew. Math. 556\/} (2003), 1--46.

\bibitem{AGS:06}
{\sc {\'A}lvarez-C{\'o}nsul, L., Garc{\'i}a-Prada, O., and Schmitt, A. H.~W.}
\newblock On the geometry of moduli spaces of holomorphic chains over compact
  {R}iemann surfaces.
\newblock {\em IMRP Int. Math. Res. Pap.\/} (2006), Art. ID 73597, 82.

\bibitem{AGP:16}
{\sc {Andersen}, J.~E., {Gukov}, S., and {Pei}, D.}
\newblock {The Verlinde formula for {H}iggs bundles}.
\newblock {\em ArXiv e-prints 1608.01761\/} (Aug. 2016).

\bibitem{BNR:89}
{\sc Beauville, A., Narasimhan, M.~S., and Ramanan, S.}
\newblock Spectral curves and the generalised theta divisor.
\newblock {\em J. Reine Angew. Math. 398\/} (1989), 169--179.

\bibitem{BGL:11}
{\sc Biswas, I., Gothen, P.~B., and Logares, M.}
\newblock On moduli spaces of {H}itchin pairs.
\newblock {\em Math. Proc. Cambridge Philos. Soc. 151}, 3 (2011), 441--457.

\bibitem{BY:96}
{\sc Boden, H.~U., and Yokogawa, K.}
\newblock Moduli spaces of parabolic {H}iggs bundles and parabolic {$K(D)$}
  pairs over smooth curves. {I}.
\newblock {\em Internat. J. Math. 7}, 5 (1996), 573--598.

\bibitem{BGGH:17}
{\sc {Bradlow}, S., {Garc{\'i}a-Prada}, O., {Gothen}, P., and {Heinloth}, J.}
\newblock {Irreducibility of moduli of semistable chains and applications to
  U(p,q)-{H}iggs bundles}.
\newblock {\em ArXiv e-prints 1703.06168\/} (Mar. 2017).

\bibitem{BD:91}
{\sc Bradlow, S.~B., and Daskalopoulos, G.~D.}
\newblock Moduli of stable pairs for holomorphic bundles over {R}iemann
  surfaces.
\newblock {\em Internat. J. Math. 2}, 5 (1991), 477--513.

\bibitem{BGG:04}
{\sc Bradlow, S.~B., Garc{\'i}a-Prada, O., and Gothen, P.~B.}
\newblock Moduli spaces of holomorphic triples over compact {R}iemann surfaces.
\newblock {\em Math. Ann. 328}, 1-2 (2004), 299--351.

\bibitem{GGM:07}
{\sc Garc{\'i}a-Prada, O., Gothen, P.~B., and Mu\~noz, V.}
\newblock Betti numbers of the moduli space of rank 3 parabolic {H}iggs
  bundles.
\newblock {\em Mem. Amer. Math. Soc. 187}, 879 (2007), viii+80.

\bibitem{GHS:14}
{\sc Garc{\'i}a-Prada, O., Heinloth, J., and Schmitt, A.}
\newblock On the motives of moduli of chains and {H}iggs bundles.
\newblock {\em J. Eur. Math. Soc. (JEMS) 16}, 12 (2014), 2617--2668.

\bibitem{GR:15}
{\sc Garc{\'i}a-Raboso, A., and Rayan, S.}
\newblock Introduction to nonabelian {H}odge theory: flat connections, {H}iggs
  bundles and complex variations of {H}odge structure.
\newblock In {\em Calabi-{Y}au varieties: arithmetic, geometry and physics},
  vol.~34 of {\em Fields Inst. Monogr.} Fields Inst. Res. Math. Sci., Toronto,
  ON, 2015, pp.~131--171.

\bibitem{PBG:94}
{\sc Gothen, P.~B.}
\newblock The {B}etti numbers of the moduli space of stable rank {$3$} {H}iggs
  bundles on a {R}iemann surface.
\newblock {\em Internat. J. Math. 5}, 6 (1994), 861--875.

\bibitem{PBG:95}
{\sc Gothen, P.~B.}
\newblock {\em The {T}opology of {H}iggs {B}undle {M}oduli {S}paces}.
\newblock 1995.
\newblock Thesis (Ph.D.), University of Warwick.

\bibitem{PBG:01}
{\sc Gothen, P.~B.}
\newblock Components of spaces of representations and stable triples.
\newblock {\em Topology 40}, 4 (2001), 823--850.

\bibitem{GK:05}
{\sc Gothen, P.~B., and King, A.~D.}
\newblock Homological algebra of twisted quiver bundles.
\newblock {\em J. London Math. Soc. (2) 71}, 1 (2005), 85--99.

\bibitem{GN:17}
{\sc {Gothen}, P.~B., and {Nozad}, A.}
\newblock {Quiver bundles and wall crossing for chains}.
\newblock {\em ArXiv e-prints 1709.09581\/} (Sept. 2017).

\bibitem{DHL:16}
{\sc {Halpern-Leistner}, D.}
\newblock {The equivariant {V}erlinde formula on the moduli of {H}iggs bundles,
  with appendix by C. Teleman}.
\newblock {\em ArXiv e-prints 1608.01754\/} (Aug. 2016).

\bibitem{HT:03}
{\sc Hausel, T., and Thaddeus, M.}
\newblock Mirror symmetry, {L}anglands duality, and the {H}itchin system.
\newblock {\em Invent. Math. 153}, 1 (2003), 197--229.

\bibitem{NJH:86}
{\sc Hitchin, N.~J.}
\newblock The self-duality equations on a {R}iemann surface.
\newblock {\em Proc. London Math. Soc. (3) 55}, 1 (1987), 59--126.

\bibitem{NJH:87}
{\sc Hitchin, N.~J.}
\newblock Stable bundles and integrable systems.
\newblock {\em Duke Math. J. 54}, 1 (1987), 91--114.

\bibitem{HKLR:87}
{\sc Hitchin, N.~J., Karlhede, A., Lindstr{\"o}m, U., and Ro{\v c}ek, M.}
\newblock Hyper-{K}\"ahler metrics and supersymmetry.
\newblock {\em Comm. Math. Phys. 108}, 4 (1987), 535--589.

\bibitem{SM:12}
{\sc Mozgovoy, S.}
\newblock Solutions of the motivic {ADHM} recursion formula.
\newblock {\em Int. Math. Res. Not. IMRN}, 18 (2012), 4218--4244.

\bibitem{SM:16}
{\sc {Mozgovoy}, S.}
\newblock {Higgs bundles over {$\mathbb{P}^1$} and quiver representations}.
\newblock {\em ArXiv e-print 1611.08515\/} (Nov. 2016).

\bibitem{MS:17}
{\sc {Mozgovoy}, S., and {Schiffmann}, O.}
\newblock {Counting Higgs bundles and {type} A quiver bundles}.
\newblock {\em ArXiv e-prints 1705.04849\/} (May 2017).

\bibitem{HK:94}
{\sc Nakajima, H.}
\newblock Instantons on {ALE} spaces, quiver varieties, and {K}ac-{M}oody
  algebras.
\newblock {\em Duke Math. J. 76}, 2 (1994), 365--416.

\bibitem{HK:96}
{\sc Nakajima, H.}
\newblock Varieties associated with quivers.
\newblock In {\em Representation theory of algebras and related topics
  ({M}exico {C}ity, 1994)}, vol.~19 of {\em CMS Conf. Proc.} Amer. Math. Soc.,
  Providence, RI, 1996, pp.~139--157.

\bibitem{NN:91}
{\sc Nitsure, N.}
\newblock Moduli space of semistable pairs on a curve.
\newblock {\em Proc. London Math. Soc. (3) 62}, 2 (1991), 275--300.

\bibitem{SR:11}
{\sc Rayan, S.}
\newblock {\em Geometry of {C}o-{H}iggs {B}undles}.
\newblock 2011.
\newblock Thesis (Ph.D.), University of Oxford.

\bibitem{SR:13}
{\sc Rayan, S.}
\newblock Co-{H}iggs bundles on {$\Bbb P^1$}.
\newblock {\em New York J. Math. 19\/} (2013), 925--945.

\bibitem{SR:17}
{\sc Rayan, S.}
\newblock The quiver at the bottom of the twisted nilpotent cone on
  {$\Bbb{P}^1$}.
\newblock {\em Eur. J. Math. 3}, 1 (2017), 1--21.

\bibitem{AS:05}
{\sc Schmitt, A.}
\newblock Moduli for decorated tuples of sheaves and representation spaces for
  quivers.
\newblock {\em Proc. Indian Acad. Sci. Math. Sci. 115}, 1 (2005), 15--49.

\bibitem{AS:12}
{\sc Schmitt, A.}
\newblock A remark on semistability of quiver bundles.
\newblock {\em Eurasian Math. J. 3}, 1 (2012), 110--138.

\bibitem{AS:13}
{\sc Schmitt, A.}
\newblock Global boundedness for semistable decorated principal bundles with
  special regard to quiver sheaves.
\newblock {\em J. Ramanujan Math. Soc. 28A\/} (2013), 443--490.

\bibitem{AS:08}
{\sc Schmitt, A. H.~W.}
\newblock {\em Geometric invariant theory and decorated principal bundles}.
\newblock Zurich Lectures in Advanced Mathematics. European Mathematical
  Society (EMS), Z\"urich, 2008.

\bibitem{AS:17}
{\sc Schmitt, A. H.~W.}
\newblock Generically semistable linear quiver sheaves.
\newblock In {\em Functional analysis in interdisciplinary applications},
  vol.~216 of {\em Springer Proc. Math. Stat.} Springer, Cham, 2017,
  pp.~393--415.

\bibitem{TSCH:91}
{\sc Simpson, C.~T.}
\newblock The ubiquity of variations of {H}odge structure.
\newblock In {\em Complex geometry and {L}ie theory ({S}undance, {UT}, 1989)},
  vol.~53 of {\em Proc. Sympos. Pure Math.} Amer. Math. Soc., Providence, RI,
  1991, pp.~329--348.

\bibitem{TSCH:92}
{\sc Simpson, C.~T.}
\newblock Higgs bundles and local systems.
\newblock {\em Inst. Hautes \'Etudes Sci. Publ. Math.}, 75 (1992), 5--95.

\bibitem{MT:94}
{\sc Thaddeus, M.}
\newblock Stable pairs, linear systems and the {V}erlinde formula.
\newblock {\em Invent. Math. 117}, 2 (1994), 317--353.

\end{thebibliography}

\end{document}